\documentclass{amsart}
\usepackage{amsmath}
\usepackage{amssymb}
\usepackage{amsthm}
\usepackage{dsfont}
\usepackage{lineno}
\usepackage{subfigure}
\usepackage{graphicx}
\usepackage{color}
\usepackage{xspace}
\usepackage{pdfsync}
\usepackage{hyperref} 
\graphicspath{{TransportFigures/}}
\DeclareMathOperator*{\argmin}{argmin}

\newcommand{\trace}{\text{trace}}

\newcommand{\bq}{\begin{equation}}
\newcommand{\eq}{\end{equation}}
\newcommand{\R}{\mathbb{R}}

\newcommand{\Dt}{\mathcal{D}}
\newcommand{\abs}[1]{\left\vert#1\right\vert}

\newcommand{\MA}{{Monge-Amp\`ere}\xspace}

\newcommand{\G}{\mathcal{G}}
\newcommand{\bO}{\mathcal{O}}

\newcommand{\xv}{\mathbf{x}}
\newcommand{\ev}{\mathbf{e}}
\newcommand{\nv}{\mathbf{n}}
\newcommand{\xo}{\mathbf{x}_0}
\newcommand{\uk}{u^k}
\newcommand{\ukp}{u^{k+1}}
\newcommand{\ckp}{c^{k+1}}
\newcommand{\Yk}{Y^k}
\newcommand{\phik}{\phi^k}

\newcommand{\one}{\mathds{1}}

\newtheorem{theorem}{Theorem}
\theoremstyle{lemma}
\newtheorem{lemma}[theorem]{Lemma}
\newtheorem{definition}{Definition}
\theoremstyle{remark}
\newtheorem{example}{Example}
\newtheorem{remark}{Remark}

\begin{document}

\title[Numerics for Monge-Amp\`ere with transport boundary conditions]{A numerical method for the elliptic Monge-Amp\`ere equation with transport boundary conditions}

\author{Brittany D. Froese
        }
        \thanks{Department of Mathematics, Simon Fraser University, 8888 University 
        Drive, Burnaby, BC, 
        V5A 1S6, Canada ({\tt bdf1@sfu.ca}).}

\begin{abstract}
The problem of optimal mass transport arises in numerous applications including image registration, mesh generation, reflector design, and astrophysics.  One approach to solving this problem is via the Monge-Amp\`ere equation.  While recent years have seen much work in the development of numerical methods for solving this equation, very little has been done on the implementation of the transport boundary condition.  In this paper, we propose a method for solving the transport problem by iteratively solving a Monge-Amp\`ere equation with Neumann boundary conditions.  To enable mappings between variable densities, we extend an earlier discretization of the equation to allow for right-hand sides that depend on gradients of the solution [Froese and Oberman, \emph{SIAM J. Numer. Anal.}, 49 (2011) 1692--1714].  This discretization provably converges to the viscosity solution.  The resulting system is solved efficiently with Newton's method.  We provide several challenging computational examples that demonstrate the effectiveness and efficiency ($\mathcal{O}(M)-\mathcal{O}(M^{1.3})$ time) of the proposed method.
\end{abstract}

\date{\today}    
%\subjclass[1991]{}
\keywords{Monge-Amp\`ere equation, Optimal transport, Elliptic partial differential equations, Finite difference methods, Viscosity solutions 
}
\maketitle

\section{Introduction}\label{sec:intro}
In this article, we propose a method for solving the elliptic \MA equation for a convex function $u$ on a domain $X \in \R^d$ subject to the transport condition $\nabla u:X \to Y$ where $Y$ is a connected set in $\R^d$.  The method involves solving a sequence of \MA equations with Neumann boundary conditions.  We also describe an efficient finite difference method for solving these sub-problems.  To demonstrate the capabilities of this solution method, we present computational results for several challenging numerical examples, which include the recovery of inverse maps, mapping onto unbounded density functions, mapping from a disconnected domain, and mapping onto non-convex sets.

\subsection{$L^2$ optimal transport}\label{sec:optimaltransport}
The motivation for this work is the problem of optimal mass transport~\cite{Ambrosio,EvansSurvey,Villani}.  The problem originally considered by Monge is how to transport a given pile of sand into a hole in the most cost efficient way.  Monge originally considered a cost equal to the magnitude of the distance the sand must be transported.  More generally, the Monge-Kantorovich mass transport problem is to find a mapping $s(x)$ that takes the density $f(x)$ in the space $X\in\R^d$ to the density $g(y)$ in the space $Y\in\R^d$ and that minimizes the cost functional
\[ I[s] = \int\limits_X c(x,s(x))\,dx \]
where $c(x,y)$ denotes the cost of transporting a unit of mass from the point $x\in X$ to the point $y\in Y$.  Here the data must satisfy the condition that total mass is conserved:
\[ \int\limits_X f(x)\,dx = \int\limits_Y g(y)\,dy. \]

The problem of optimal mass transport arises in a number of important applications including
image registration~\cite{Haker, HakerRegistration, RehmanRegistration},
 mesh generation~\cite{Delzanno, DelzannoGrid, Budd}, 
 reflector design~\cite{GlimmOlikerReflectorDesign,GlimmOliker2Reflectors}, 
 and astrophysics (estimating the shape of the early universe)~\cite{FrischUniv}.  
 
Kantorovich contributed to the understanding of optimal transport by reformulating the problem as a linear program and describing a simple dual formulation~\cite{Kantorovich1,Kantorovich2}.  While this has made many theoretical questions easier to answer, this approach also effectively doubles the dimension of the problem.  Consequently, computing the solution to even a small-scale problem is prohibitively expensive.  This motivates the development of more sophisticated methods that will enable the efficient computation of optimal maps.
 
In the special case of the quadratic cost function
\[ c(x,y) = \frac{1}{2}\abs{x-y}^2, \]
the problem has a special structure.  In this situation, the optimal mapping $s(x)$ can be expressed as the gradient of a convex function~\cite{EvansSurvey,Rockafellar}.  The problem of obtaining the optimal mapping is then equivalent to the problem of solving a fully nonlinear partial differential equation (PDE) known as the elliptic \MA equation
\bq
\label{eq:MA}\tag{MA}
\det(D^2u(x)) = f(x)/g(\nabla u(x)), \quad x\in X
\eq
subject to the transport condition
\bq
\label{eq:bc}\tag{BC}
\nabla u:X\to Y
\eq
and the convexity constraint
\bq
\label{eq:conv}\tag{C}
u \text{ is convex.}
\eq

\begin{remark}
Any solution method for the \MA equation must enforce the convexity constraint, which is necessary to ensure a unique solution.  
\end{remark}

\subsection{Related works}\label{sec:related}
In the past few years, the numerical solution of the \MA equation has received quite a bit of attention.  However, most of the available methods enforce Dirichlet, Neumann, or periodic boundary conditions rather than the transport condition that arises in many applications.

An early work by Oliker and Prussner~\cite{olikerprussner88} presented a method that converges to the Aleksandrov solution of the \MA equation in two dimensions.  Another convergent two-dimensional method was described by Oberman~\cite{ObermanEigenvalues}; this discretization converges to the viscosity solution of the equation.  Other recent methods, which perform best when solutions are sufficiently regular, have been developed by Dean and Glowinski~\cite{DGaug,DGnum2008, GlowinksiICIAM} and Feng and Neilan~\cite{FengMA, FengFully}. 

Recently, the author, together with co-authors, has extended the work of Oberman~\cite{ObermanSINUM,ObermanEigenvalues} to construct finite difference solvers that converge to the viscosity solution in any spatial dimension~\cite{BeFrObMA,FOTheory,FONum}.  These methods perform quickly even in the most singular examples.

Much less work has been done on the implementation of the transport boundary condition~\eqref{eq:bc}.  A fluid flow approach was introduced by Benamou and Brenier~\cite{BenBren} and has been further developed by Haber, Rehman, and Tannenbaum~\cite{HaberTransport}.  However, this approach is computationally expensive as it requires introducing an additional dimension to the problem.  We also mention the work by Finn, Delzanno, and Chac{\'o}n~\cite{DelzannoGrid}, which enables the mapping of a square to a region with four (possibly curved) sides.  For the related problem of optimal transport (or partial transport) with cost given by the distance $c(x,y) = \abs{x-y}$, a finite element method has been constructed by Barrett and Prigozhin~\cite{BarPrigL1Transport,BarPrigL1PartialTransport}.

\subsection{Contents}\label{sec:contents}
In \autoref{sec:analysis} of this work, we review some analysis---including weak solutions and regularity results for the \MA equation---that inform the approach taken in this paper.  In \autoref{sec:transportbc}, we describe our method for implementing the transport boundary conditions.  In \autoref{sec:discretize}, we provide a discretization of the \MA equation and prove that it converges to the viscosity solution.  In \autoref{sec:implementation}, we provide further details about the numerical implementation of our method.  In \autoref{sec:resultsMA}, we provide computational results that test our \MA solver.  In \autoref{sec:resultsTransport}, we provide computational results for several challenging and representative transport problems.  In \autoref{sec:conclude}, we summarize the main contributions of this work.

\section{Analysis and weak solutions}\label{sec:analysis}
In this section, we review some regularity results for the \MA equation that are needed to fully explain the approach taken in this work.

\subsection{Weak solutions}\label{sec:weak}
The \MA equation is a second order PDE, so classical solutions of this equation should have at least two continuous derivatives.  However, these classical $C^2$ solutions do not always exist.  Thus it is necessary to use some notion of weak solution to make sense of non-smooth solutions of the \MA equation.  

Weak solutions of the \MA equation can be defined in different ways.  The discretization used in this work is motivated by the \emph{viscosity solution} (which, in most cases, is equivalent to the more general \emph{Aleksandrov solution}).

We recall the definition of viscosity solutions~\cite{CIL}, which are defined for the \MA equation in~\cite{Gutierrez}. 
\begin{definition}
Let $u \in C(X)$ be convex and $f\geq0$ be continuous.  The function $u$ is a \emph{viscosity subsolution (supersolution)} of the \MA equation in $X$ if whenever convex $\phi\in C^2(X)$ and $x_0\in X$ are such that $(u-\phi)(x)\leq(\geq)(u-\phi)(x_0)$ for all $x$ in a neighborhood of $x_0$, then we must have
\[ \det(D^2\phi(x_0)) \geq(\leq)f(x_0). \]
The function $u$ is a \emph{viscosity solution} if it is both a viscosity subsolution and supersolution.
\end{definition}

\begin{example}[Viscosity solution of \MA]
For concreteness, we provide a particular example of a function that, though not a classical $C^2$ solution of the \MA equation, can be understood as a viscosity solution.  Consider~\eqref{eq:MA} with solution and $f$ given by
\[ 
u(\xv) = \frac{1}{2}((\abs{\xv}-1)^+)^2, 
\qquad
f(\xv) = (1-1/\abs{\xv})^+. \]
This function $u$ (pictured in \autoref{fig:viscosity}) is a viscosity solution---but not a classical $C^2$ solution---of the \MA equation.

We verify the definition of a viscosity solution.  This only needs to be done at points where $\abs{\xo}=1$ (since $u$ is locally $C^2$ away from this circle).  We note that $f$ is equal to zero on this circle.

We begin by checking convex $C^2$ functions $\phi \leq u$ with $\phi(\xo) = u(\xo) = 0$ (that is, $u-\phi$ has a local minimum here).  
Since $\nabla u(\xo) = 0$, we require $\nabla \phi(\xo) = 0$ as well.  Since $u$ is constant in part of any neighborhood of $\xo$, any convex $\phi$ must also be constant in this part of the neighborhood in order to ensure that $u-\phi$ has a local minimum.  This means that $\phi$ has zero curvature in some directions, so that $\det D^2 \phi(\xo) = 0$, as required by the definition of the viscosity solution.  We conclude that $u$ is a supersolution of the \MA equation.

We also need to check functions $\phi \geq u$ with $\phi(\xo) = u(\xo) = 0$ (so that $u-\phi$ has a local maximum).  Since $\phi$ is convex, it will automatically satisfy the condition $\det D^2\phi(\xo) \geq 0$ and we conclude that $u$ is a subsolution. 
\end{example}

\begin{figure}[htdp]
	\centering
        {\includegraphics[width=.4\textwidth]{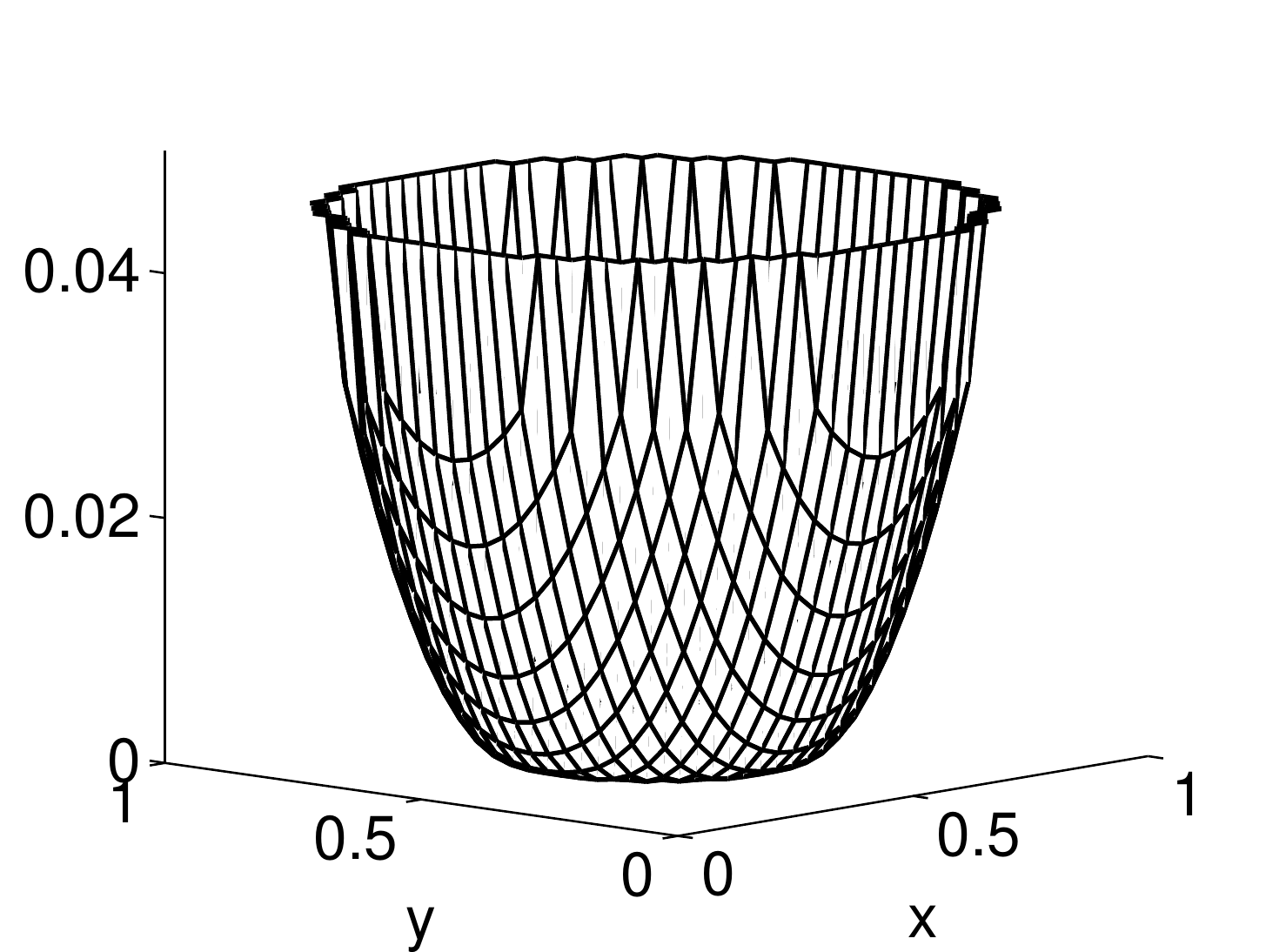}}
  	\vspace*{-10pt}\caption{A viscosity solution of the \MA equation.}
  	\label{fig:viscosity}
\end{figure}

\subsection{Regularity}\label{sec:reg}
We now review the regularity we can expect for solutions of the $L^2$ optimal transport problem.  These results are due to Caffarelli~\cite{CafBdyReg,CafRegMap,CafBdyReg2}.

We begin by noting that, as with the \MA equation, solutions of the transport problem need not be smooth.  An example of a singular solution (see \autoref{fig:singularmap}) is the problem of mapping the circle
\[ X = \{(x_1,x_2)\mid x_1^2+x_2^2<1\} \]
onto the disconnected set
\begin{multline*} Y = \{(x_1,x_2)\mid x_1 < -0.25,(x_1+0.25)^2+x_2^2 < 1 \} \\ \cup \{(x_1,x_2)\mid x_1 > 0.25,(x_1-0.25)^2+x_2^2 < 1 \}.\end{multline*}

In fact, the solution remains singular even if the disconnected region $Y$ is approximated by a connected region $Y_\epsilon$.

While we do not solve the problem of mapping onto a disconnected region, we are able to solve for the inverse mapping (which takes the disconnected set $Y$ to the connected set $X$) in~\S\ref{sec:exSplitCircle}.  

\begin{figure}[htdp]
	\centering
        {\includegraphics[width=0.9\textwidth]{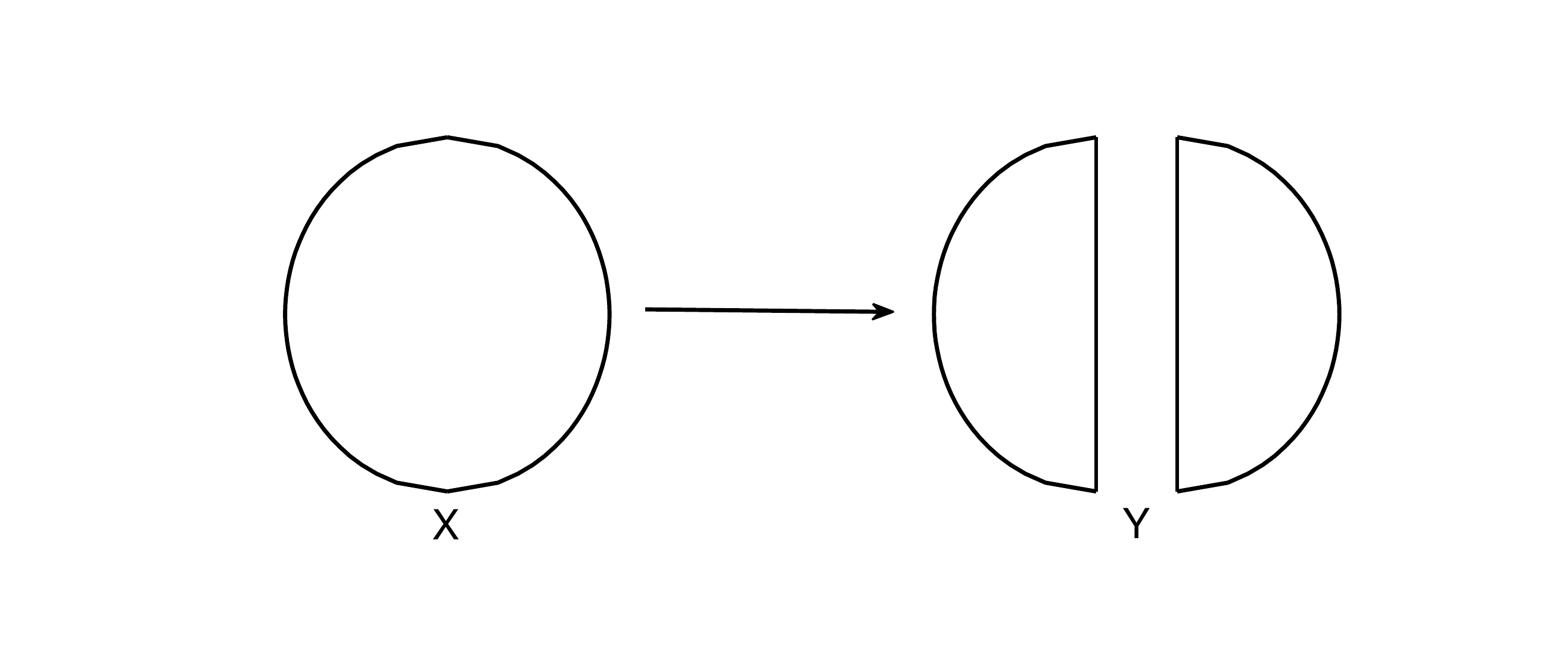}}
  	\vspace*{-30pt}\caption{A transport problem with a singular solution.}
  	\label{fig:singularmap}
\end{figure} 

As long as the sets $X$,$Y$ are bounded, we are at least guaranteed that the solution of the \MA equation is differentiable almost everywhere with bounded gradient.

\begin{remark}
When the solution to the \MA equation is not differentiable, the map is given by the sub-gradient rather than the gradient.  This allows a single point to be mapped onto a region rather than a single point.
\end{remark}

More regularity is guaranteed if we restrict ourselves to convex target sets $Y$.
\begin{theorem}[Interior Regularity]\label{thm:intreg}
Suppose that $X,Y$ are bounded, connected, open sets and $Y$ is convex.  Suppose also that the density functions
\[ f:X\to(0,+\infty),\, g:Y\to(0,+\infty) \]
are bounded away from 0 and $+\infty$.  Then the solution of the \MA equation~\eqref{eq:MA},~\eqref{eq:bc},~\eqref{eq:conv} belongs to $C^{1,\alpha}_{loc}(X)$ for some $0<\alpha<1$.

If, in addition, the density functions $f,g\in C^\beta$ for some $0<\beta<1$ then the solution of \MA belongs to $C^{2,\alpha}_{loc}(X)$ for every $0<\alpha<\beta$.
\end{theorem}

If both sets $X,Y$ are uniformly convex, we can obtain regularity up to the boundary as well.
\begin{theorem}[Boundary Regularity]\label{them:bdyreg}
Suppose, in addition to the hypotheses of \autoref{thm:intreg}, that the sets $X$ and $Y$ are uniformly convex.  Then the solution of \MA is in $C^{2,\alpha}(\bar{X})$ for some $0<\alpha<1$.
\end{theorem}

\section{Transport boundary conditions}\label{sec:transportbc}
In this section, we discuss the transport boundary conditions in more detail.  We describe a method for solving this challenging problem by solving a sequence of more tractable sub-problems; these are \MA equations subject to Neumann boundary conditions.

\subsection{Nonlinear boundary conditions}\label{sec:nonlinbc}
In the problem of $L^2$ optimal transport between convex sets $X,Y\in\R^d$, the transport condition~\eqref{eq:bc}
\[ \nabla u:X\to Y, \]
also known as the second boundary value problem, can be enforced by simply requiring boundary points to map to boundary points~\cite{Pogorelov,TrudWang2ndBVP,Urbas2ndBVP}:
\[ \nabla u:\partial X \to \partial Y. \]
In particular, if the boundary of the region $Y$ is defined by the function
\[ \Phi(y) = 0, \]
we can write the transport boundary condition as
\bq\label{eq:BC2} \Phi(\nabla u(x)) = 0, \quad x\in\partial X. \eq
While we might try simply enforcing this nonlinear equation at boundary points, the function $\phi$ can be highly nonlinear and non-smooth.  As a result, it will be difficult to construct a discretization that is consistent with a possibly singular solution of the equation and that will permit fast solvers to remain stable.

\subsection{Mapping to rectangles}\label{sec:rect}
The situation simplifies significantly if we are simply mapping a rectangle to a rectangle.  In this case, since the optimal $L^2$ mapping does not permit twisting or rotation, we expect the four sides of the rectangle $X$ to map to the corresponding sides of the rectangle $Y$.  

As a concrete example (see \autoref{fig:squares}), suppose that the sets $X,Y\in\R^2$ are defined as
\[ X = (0,1)\times(0,1),\quad Y = (0,1)\times(0,1). \]
Then, for example, we expect the function $\nabla u(x)$ to map the segment $x_1 = 0,\,x_2\in[0,1]$ to the segment $y_1=0,\,y_2\in[0,1]$.  That is,
\[ u_{x_1}(0,x_2) = 0. \]
Similarly, we will have
\[ u_{x_1}(1,x_2) = 1,\quad u_{x_2}(x_1,0) = 0,\quad u_{x_2}(x_2,1) = 1. \]
This is simply a (linear) Neumann boundary condition, which is straightforward to implement~\cite{BSnum,ObermanSINUM}.

\begin{figure}[htdp]
	\centering
        {\includegraphics[width=\textwidth]{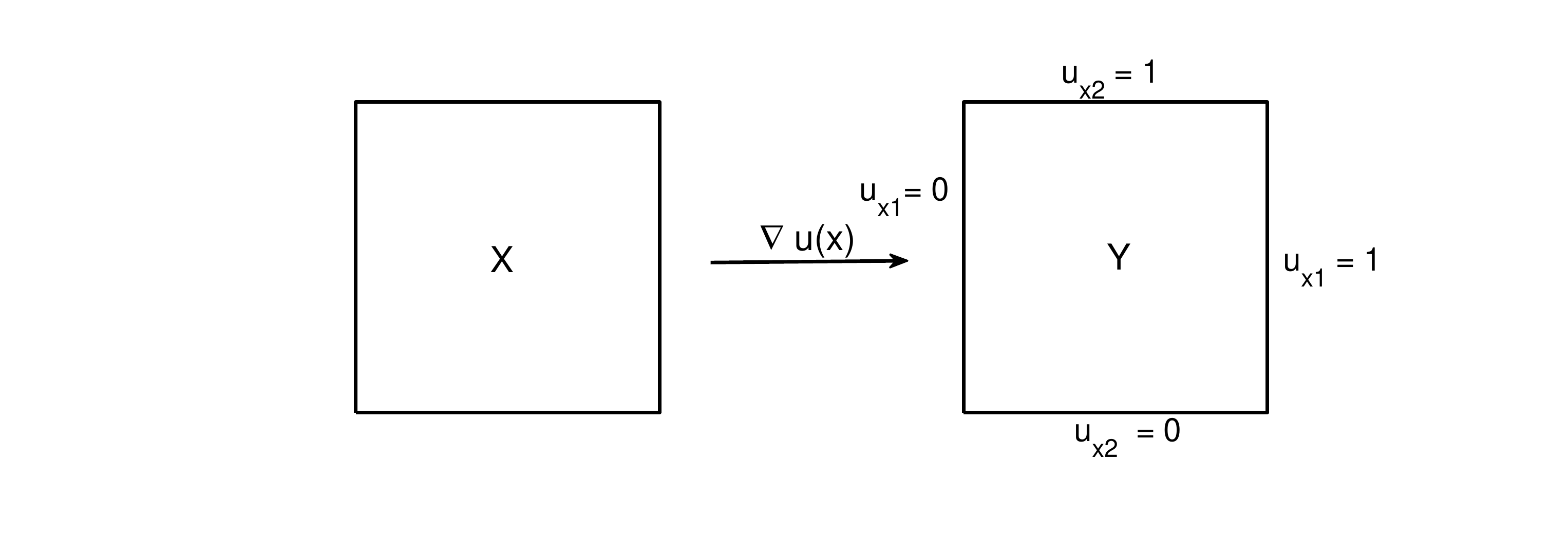}}
  	\vspace*{-40pt}\caption{Mapping between squares.}
  	\label{fig:squares}
\end{figure}

\subsection{A sequence of Neumann boundary conditions}\label{sec:neumann}
Given the appearance of the gradient in the transport boundary condition~\eqref{eq:BC2} and the simplicity of implementing a Neumann boundary condition, we would like to find the Neumann boundary condition 
\[  \frac{\partial u}{\partial \nv} = \phi(x),\quad x\in\partial X \]
for the \MA equation that is equivalent to solving the more challenging problem~\eqref{eq:MA},~\eqref{eq:bc},~\eqref{eq:conv}.   Here the vector $\nv$ refers to the unit outward normal vector at each point $x\in\partial X$.

It is not at all apparent from~\eqref{eq:bc} what the equivalent Neumann boundary condition should be. However, we suggest a sequence of Neumann boundary conditions that can be used to numerically determine the correct function $\phi$.

We first recall that the gradient of the exact solution $u$ maps the boundary of the set $X$ to the boundary of $Y$
\[ \nabla u:\partial X\to \partial Y \]
and that the correct Neumann condition is given by
\[ \phi(x) = \nabla u(x)\cdot \nv(x),\quad x\in\partial X. \]

To find this function, we suppose that we have a convex approximation $\uk$ to the solution of the \MA transport problem.  Then the (sub-)gradient of this function will map the domain $X$ onto some set $\Yk\in\R^d$ and, since $\uk$ is convex,
\[ \nabla\uk:\partial X\to\partial\Yk. \]
In reality, we would like the image of the gradient to be $\partial Y$, the boundary of the target set.  This motivates us to consider the projection of $\partial\Yk = \nabla \uk(\partial X)$ onto the correct set of boundary points $\partial Y$:
\[ \text{Proj}_{\partial Y} (\nabla\uk(x)) = \argmin\limits_{y\in\partial Y}{\|y - \nabla \uk(x)\|_2^2},\quad x\in\partial X. \]

From this we extract a new Neumann boundary condition
\[ \phik(x) = \text{Proj}_{\partial Y} (\nabla\uk(x))\cdot\nv \]
and solve the \MA equation once again with this updated boundary condition to obtain a new approximation $\ukp$.

To summarize, we iterate to produce a sequence of functions $(u^1,u^2,\ldots)$ obtained by solving the \MA equation
\bq\label{eq:it1}
\begin{cases}
\det(D^2\ukp(x)) = f(x)/g(\nabla\ukp(x)), & x\in X\\
\nabla\ukp(x)\cdot\nv(x) = \text{Proj}_{\partial Y} (\nabla\uk(x))\cdot\nv \equiv \phik(x), & x\in\partial X\\
\ukp \text{ is convex.}
\end{cases}
\eq

We make the important observation that these boundary conditions do \emph{not} pin down the values of $\nabla\ukp$ on the boundary.  This would be a mistake since we know only that $\nabla u:\partial X\to\partial Y$ and not the exact values of $\nabla u(x)$ on the boundary.  Instead, each Neumann condition fixes only one component of the gradient (the normal component) and allows the remaining component(s) to slide as needed to ensure that the \MA equation is satisfied.

\subsection{Solvability of sub-problems}\label{sec:solvability}
We note that the iteration~\eqref{eq:it1} may not be well-posed.  The problem here is that, while the \MA equation with the correct Neumann values $\phi(x)$ has a solution, the sub-problems we have described may not be solvable.

We recall that for the \MA equation with a Neumann condition:
\[
\begin{cases}
\det(D^2u) = f(x)/g(\nabla u(x)), & x\in X\\
\nabla u(x)\cdot\nv(x) = \psi(x), & x\in\partial X\\
u \text{ is convex},
\end{cases}
\]
a solution (unique up to an additive constant) exists only if an implicit solvability condition is satisfied~\cite{LionsNeumannMA}.

To get around this problem, we instead solve a problem of the form
\[
\begin{cases}
\det(D^2u) = c f(x)/g(\nabla u(x)), & x\in X\\
\nabla u(x)\cdot\nv(x) = \psi(x), & x\in\partial X\\
u \text{ is convex},\\
\int_X u(x)\,dx = 0
\end{cases}
\]
for the unknowns $c>0$ and $u(x)$, where the constant $c$ is chosen to ensure the equation has a solution and the mean-zero condition forces the solution to be unique (instead of unique up to an additive constant).

Of course, if we are given the correct Neumann values $\phi(x)$ for the solution to the transport problem, the constant $c$ will simply be equal to one.  However, by relaxing this condition we make it possible to solve the sub-problems when the solvability condition requires $c$ to be slightly different than one.

To summarize, we solve the transport problem by performing the iteration
\bq\label{eq:it2}
\begin{cases}
\det(D^2\ukp(x)) = \ckp f(x)/g(\nabla\ukp(x)), & x\in X\\
\nabla\ukp(x)\cdot\nv(x) = \text{Proj}_{\partial Y} (\nabla\uk(x))\cdot\nv(x) \equiv \phik(x), & x\in\partial X\\
\ukp \text{ is convex,}\\
\int_X \ukp(x)\,dx = 0.
\end{cases}
\eq

\begin{remark}
As we pointed out in~\S\ref{sec:reg}, solutions to the optimal transport problem need not be continuously differentiable up to the boundary.  In this case, the Neumann boundary condition should be understood in the viscosity sense~\cite{LionsNeumannHJ}.  To obtain the boundary condition $\phik(x)$ at a point where the gradient $\nabla\uk(x)$ is not defined, we instead look at the projection of a value in the sub-gradient of $\uk(x)$.
\end{remark}

\section{Discretization}\label{sec:discretize}
Having described an iteration for solving the transport problem, we now need to describe the method for solving the sub-problems.  The most challenging step here is to discretize the \MA equation.  In this section, we describe a discretization that provably converges to the viscosity solution.

\subsection{Standard finite difference methods}\label{sec:standardFD}
The simplest thing to do is to simply discretize the \MA equation using standard centered differences.  

In two dimensions, for example, the \MA equation has the form
\[ u_{x_1x_1}u_{x_2x_2}-u_{x_1x_2}^2 = f(x)/g(u_{x_1},u_{x_2}). \]
A standard centered difference discretization of this equation is
\bq\label{eq:discStand} 
MA^h_S[u] = (\Dt_{x_1x_1}u)(\Dt_{x_2x_2}u) - (\Dt_{x_1x_2}u)^2 - f(x)/g(\Dt_{x_1}u,\Dt_{x_2}u)
\eq
where the finite difference operators are defined by
\begin{align*}
[\Dt_{x_1x_1}u]_{ij} &= \frac{1}{h^2} 
\left(
{u_{i+1,j}+u_{i-1,j}-2u_{i,j}}
\right)
\\
[\Dt_{x_2x_2}u]_{ij} &= \frac{1}{h^2}
\left(
u_{i,j+1}+u_{i,j-1}-2u_{i,j}
\right)
\\
[\Dt_{x_1x_2}u]_{ij} &= \frac{1}{4h^2}
\left(
u_{i+1,j+1}+u_{i-1,j-1}-u_{i-1,j+1}-u_{i+1,j-1}
\right)
\\
[\Dt_{x_1}u]_{ij} &= \frac{1}{2h}
\left(
u_{i+1,j}-u_{i-1,j}
\right)\\
[\Dt_{x_2}u]_{ij} &= \frac{1}{2h}
\left(
u_{i,j+1}-u_{i,j-1}
\right).
\end{align*}

However, as is pointed out in~\cite{FONum}, this discretization may fail to converge to the correct viscosity solution and solution methods can become unstable.

\subsection{Convergent finite difference methods}\label{sec:fd}
Before we discuss our discretization of the \MA equation, we briefly review the theory of convergent finite difference methods for viscosity solutions of nonlinear elliptic equations.  The foundation for this is a result by Barles and Souganidis~\cite{BSnum}.

\begin{theorem}[Convergence of Approximation Schemes]
Consider a degenerate elliptic equation, for which there exist unique viscosity solutions.  A consistent, stable approximation scheme converges uniformly on compact subsets to the viscosity solution, provided it is monotone.
\end{theorem}

Oberman~\cite{ObermanSINUM} used this result to further characterize convergent finite difference discretizations.
We recall that a finite difference equation has the form
\[ F^i[u] = F^i(u_i,u_i-u_j|_{j \neq i}). \]
Then a degenerate elliptic (monotone) scheme can be defined as follows:
\begin{definition}\label{def:deg}
The scheme $F$ is \emph{degenerate elliptic} if $F^i$ is non-decreasing in each variable.
\end{definition}
\begin{theorem}[Convergence of Finite Difference Discretizations]
Consider a degenerate elliptic equation, for which there exist unique viscosity solutions.  The solution to a consistent, degenerate elliptic finite difference scheme converges uniformly on compact subsets to the viscosity solution.
\end{theorem}

Even for linear elliptic equations, it is not always possible to construct monotone schemes using a narrow stencil, as was demonstrated in the early work by Motzkin and Wasow~\cite{MotzkinWasow}.  Oberman~\cite{ObermanEigenvalues} used wide stencils to construct monotone discretizations of second directional derivatives for directions $\nu$ lying on the grid.  These derivatives can be discretized using centered differences:
\bq\label{eq:ws} 
\Dt_{\nu\nu}u_i = \frac{1}{\abs{\nu}^2h^2}\left(u(x_i + \nu h) + u(x_i - \nu h) - 2u(x_i)\right). 
\eq
Depending on the direction $\nu$, this may involve a wide stencil.  At points near the boundary of the domain, some values required by the wide stencil will not be available~(\autoref{fig:stencil}).  In these cases, we use interpolation at the boundary to construct a (lower accuracy) stencil for the second directional derivative; see~\cite{ObermanEigenvalues} for more details.  The consistency error of these approximations depends on both the spatial resolution $h$ and angular resolution $d\theta$ of the stencil.

\begin{figure}
  \centering
  \subfigure[In the interior.]{
  \includegraphics[height=0.3\textwidth]
  {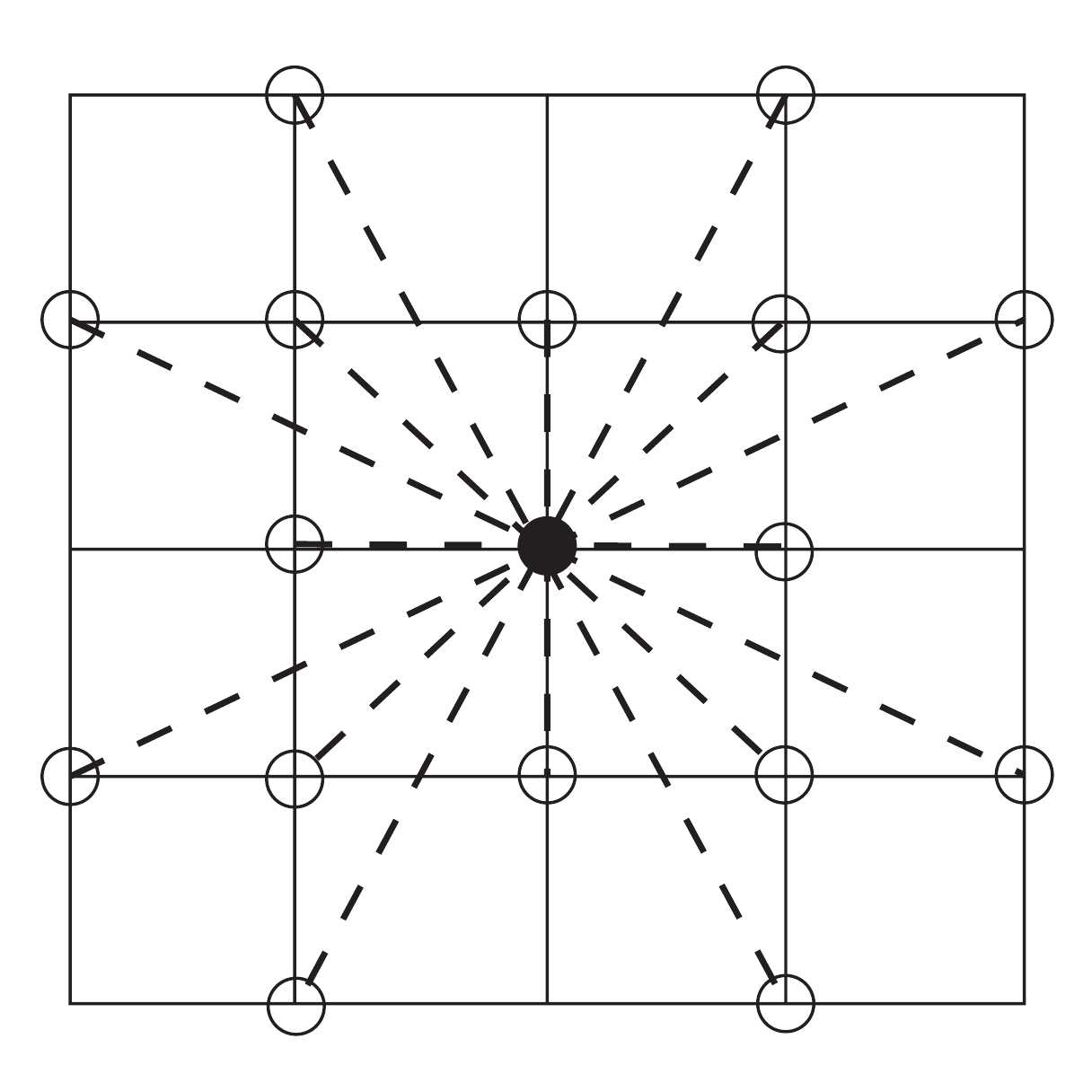}
  }                
  \subfigure[Near the boundary.]{
  \includegraphics[height=0.3\textwidth]
  {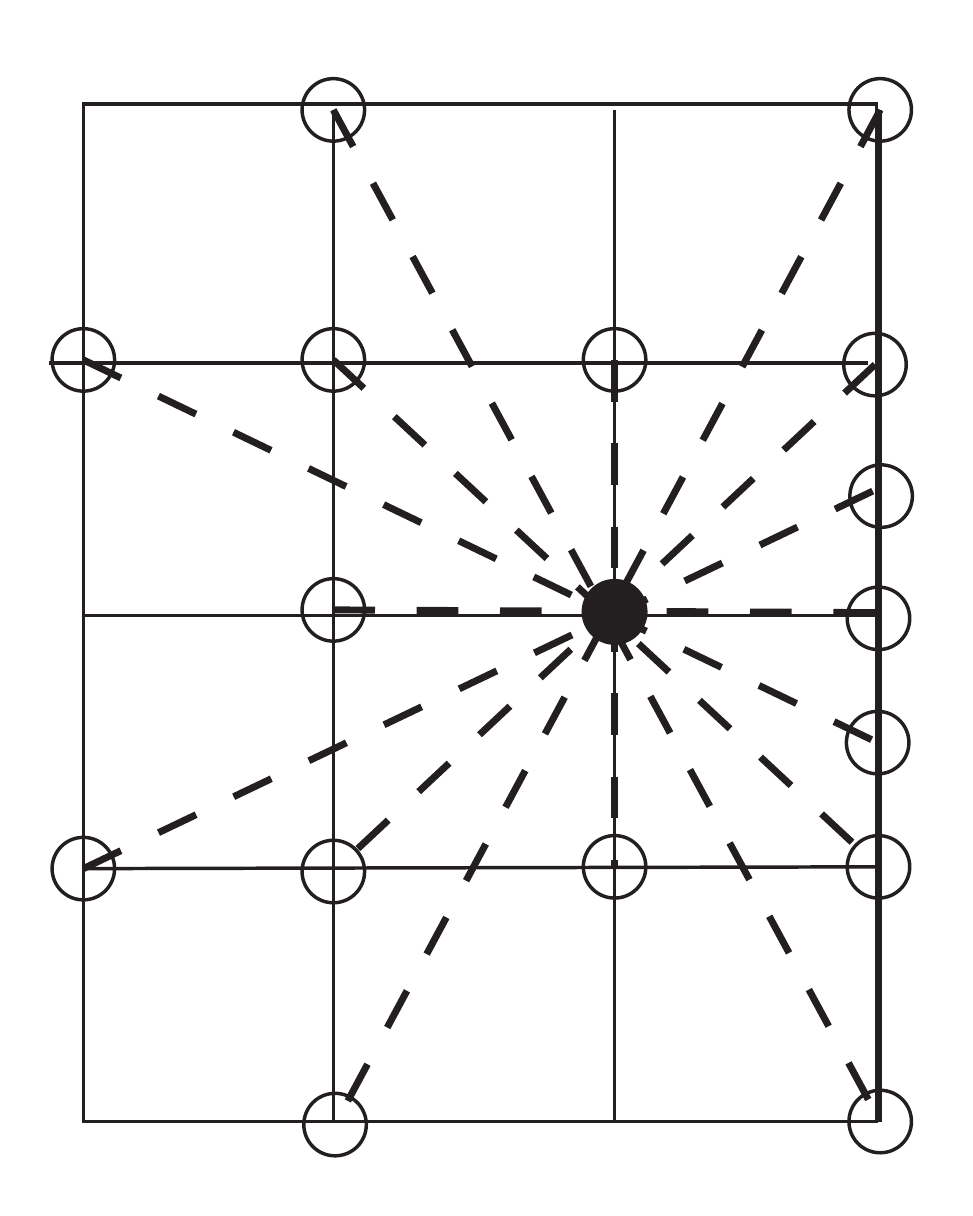}
  }
\vspace*{-12pt}\caption{
Wide stencils on a two dimensional grid.
}
\label{fig:stencil}
\end{figure}

\subsection{Discretization of the \MA operator}\label{sec:discMAop}
Next we describe a monotone discretization of the \MA operator
\[ \det(D^2u), \]
which was introduced in~\cite{FOTheory}.  This discretization is a consequence of the following variational characterization of the \MA operator
\[ \det(D^2u) = \min\limits_{(\nu_1,\ldots,\nu_d)\in V} \prod\limits_{j=1}^d \max\{u_{\nu_j\nu_j},0\} \]
where $u$ is a convex function and $V$ is the set of all orthonormal bases of $\R^d$:
\[ V = \{(\nu_1,\ldots,\nu_d)  \mid \nu_j\in\R^d,\nu_i\perp\nu_j \text{ if }i\neq j, \|\nu_j\|_2 = 1\}. \]

We are going to add an additional term to this expression in order to further penalize non-convexity:
\[ \text{det}^+(D^2u) =  \min\limits_{(\nu_1,\ldots,\nu_d)\in V} \left\{\prod\limits_{j=1}^d \max\{u_{\nu_j\nu_j},0\} + \sum\limits_{j=1}^d \min\{u_{\nu_j\nu_j},0\}\right\}.\]
We note that if $u$ is a convex function, the terms in the summation will vanish since all second directional derivatives $u_{\nu_j\nu_j}$ are non-negative.  On the other hand, if $u$ is a non-convex function then at least one of the directional derivatives $u_{\nu\nu}$ will be negative and we will have
\[  \text{det}^+(D^2u) \leq  u_{\nu\nu} < 0.\]
Thus this non-convex function cannot be a solution of the \MA equation
\bq\label{eq:MAconv} \text{det}^+(D^2u) = F(x,\nabla u)\geq0. \eq

In order to discretize this, we restrict our attention to a finite set of orthogonal vectors $\G$ on the grid.  A monotone discretization of the \MA operator is given by
\[
\min\limits_{\{\nu_1\ldots\nu_d\}\in \G}
\left\{\prod\limits_{j=1}^{d} \max\{\Dt_{\nu_j\nu_j}u,0\}  + \sum\limits_{j=1}^d \min\{\Dt_{\nu_j\nu_j}u,0\}\right\}.
\]
where $\Dt_{\nu\nu}u$ is the discretization of the second directional derivative that is defined in~\eqref{eq:ws}.

\subsection{Discretization of functions of the gradient}\label{sec:discGrad}
In the works~\cite{FOTheory,FONum}, we considered problems where the right-hand side $f$ was a function of $x$ only.  However, in the \MA equation that arises in the transport problem, the right-hand side also depends on the gradient of the solution.  That is, the equation~\eqref{eq:MA} is of the form
\bq\label{eq:gMA} \det(D^2u(x)) = F(x,\nabla u(x)). \eq
Consequently, we need to discretize not only the eigenvalues of the Hessian but also the gradient.

The simplest approach would be to simply use standard centered differences for the first derivatives:
\[ \Dt_{x_j}u(\xv) = \frac{1}{2h}(u(\xv+h\ev_j) - u(\xv-h\ev_j)) \]
where $\ev_j$ is the vector whose $i^{th}$ component is equal to the Kronecker delta $\delta_{ij}$.  While this discretization is consistent with $C^2$ solutions of the \MA equation, it is not monotone and there is no guarantee that it will converge to the viscosity solution.

Oberman~\cite{ObermanSINUM} provided some examples illustrating the construction of monotone discretizations for functions of the gradient.  For example, that work describes a monotone discretization of the absolute value of a first derivative:
\[ \abs{u_x(x_j)} = \frac{1}{h}\max\{u(x_j)-u(x_{j-1}),u(x_{j+1})-u(x_j),0\} + \bO(h). \]

For more general functions of the gradient, one approach to producing a monotone discretization is to simply use centered differences and add on a small multiple of the laplacian:
\[ g(u_x) = g(\Dt_xu) + hK_g\Dt_{xx}u + \bO(h).\]
Here $K_g$ is the Lipschitz constant of the function $g$.

However, instead of adding an additional term to the discretized equation, we could make use of the second derivatives that are already present in the \MA equation.  This is the subject of the following section.

\subsection{Discretization of the \MA equation}\label{sec:discMA}
So far we have attempted to produce a monotone discretization for each individual term in the \MA equation.  
%A consequence of this is reduced accuracy since a monotone discretization of a function of the gradient will not be more than first-order ($\bO(h)$) accurate.  
As an alternative to this, we suggest using a wide stencil to produce a discretization of the \MA equation which, though it may not be monotone for each of the individual terms, is monotone when considered as a whole.  This discretization also ensures that the linear systems that must be solved in the implementation of Newton's method are well-conditioned even when the eigenvalues of the Hessian are close to zero or the iteration is initialized poorly.

To accomplish this, we make use of the second directional derivatives $u_{\nu_j\nu_j}$ that are already present in the \MA equation, as noted in~\S\ref{sec:discGrad}.  By making a change of coordinates, we can write the gradient
\[ \nabla u =  \left(u_{x_1},\ldots,u_{x_d}\right)\]
in terms of first derivatives in the directions $\nu_j$:
\[ \tilde\nabla u = \left(u_{\nu_1},\ldots,u_{\nu_d}\right). \]
Once this is done, the only problem we might have will be if one of the second derivatives vanishes, in which case the \MA operator will not be uniformly elliptic.  We can remedy this by simply regularizing the maximum and minimum functions slightly to bound them away from zero:
\[ \max\{\cdot,0\},\min\{\cdot,0\} \to  \max\{\cdot,\delta\},\min\{\cdot,\delta\}\]
where $\delta>0$ is a small parameter.

To accomplish all this, we first need to rewrite the gradient in terms of the new coordinate system.  We consider any set of $d$ orthogonal vectors in $\R^d$: $(v_1,\ldots,v_d)$.  Now we can rewrite the gradient of a function $u$ in terms of directional derivatives along these axes:
\[ \nabla u = \left(u_{x_1},\ldots,u_{x_d}\right) = \left(\sum\limits_{j=1}^d\frac{v_j\cdot \ev_1}{\abs{v_j}}u_{v_j},\ldots,\sum\limits_{j=1}^d\frac{v_j\cdot \ev_d}{\abs{v_j}}u_{v_j}\right). \]
This enables us to discretize the gradient using a wide stencil by discretizing the directional derivative in the direction $v_j$ as
\bq\label{eq:discGrad}
\Dt_{v_j}u_i = \frac{1}{2\abs{v_j}h}\left(u(x_i + v_j h) -u(x_i - v_j h)\right), 
\eq
which has an accuracy of $\bO(h^2)$.  Near the boundary, where some of the required values may not be available, we can simply use a first-order accurate forward or backward difference.  We stress again that this discretization of the gradient is valid for \emph{any} set of orthogonal vectors $v_1,\ldots,v_d$.

Using this characterization of the gradient,  we can rewrite the \MA equation as
\begin{align*}
MA[u]&= 
	\min\limits_{(\nu_1,\ldots,\nu_d)\in V}\left\{ \prod\limits_{j=1}^d \max\{u_{\nu_j\nu_j},0\} + \sum\limits_{j=1}^d\min\left\{u_{\nu_j\nu_j},0\right\}\right\}-F(x,\nabla u)\\
	&=\min\limits_{(\nu_1,\ldots,\nu_d)\in V}\left\{ \prod\limits_{j=1}^d \max\{u_{\nu_j\nu_j},0\} + \sum\limits_{j=1}^d\min\left\{u_{\nu_j\nu_j},0\right\} - F(x,\nabla u)\right\}\\
	&\phantom{=}\min\limits_{(\nu_1,\ldots,\nu_d)\in V}\left\{ \prod\limits_{j=1}^d \max\{u_{\nu_j\nu_j},0\} + \sum\limits_{j=1}^d\min\left\{u_{\nu_j\nu_j},0\right\} \right.\\
	&\phantom{=\min\limits_{(\nu_1,\ldots,\nu_d)\in V}}- \left.F\left(x,\sum\limits_{j=1}^d\frac{\nu_j\cdot \ev_1}{\abs{\nu_j}}u_{\nu_j},\ldots,\sum\limits_{j=1}^d\frac{\nu_j\cdot \ev_d}{\abs{\nu_j}}u_{\nu_j}\right)\right\}\\
  &= \min\limits_{(\nu_1,\ldots,\nu_d)\in V} G_{(\nu_1,\ldots,\nu_d)}.
\end{align*}

As we have already described in~\eqref{eq:ws},\eqref{eq:discGrad}, the directional first and second derivatives can be discretized using a wide stencil by limiting the set of possible directions in the set $V$ to a finite set $\G$ of orthogonal vectors that lie on the grid.  We also introduce a small parameter $\delta>0$ in order to bound the maximum and minimum functions away from zero:
\[ \max\{\cdot,0\},\min\{\cdot,0\} \to  \max\{\cdot,\delta\},\min\{\cdot,\delta\}.\]

We can now define the discretization of the \MA equation as
\bq\label{eq:MAmonDisc}
MA_M^{h,d\theta,\delta}[u] 
  = \min\limits_{(\nu_1,\ldots,\nu_d)\in \G}G_{(\nu_1,\ldots,\nu_d)}^{h,d\theta,\delta}[u]
\eq
where each of the $G_{(\nu_1,\ldots,\nu_d)}^{h,d\theta,\delta}[u]$ is defined as
\bq\label{eq:monterm}
\begin{split}
G_{(\nu_1,\ldots,\nu_d)}^{h,d\theta,\delta}[u] = &\prod\limits_{j=1}^d\max\{\Dt_{\nu_j\nu_j}u,\delta\} + 
\sum\limits_{j=1}^d\min\{\Dt_{\nu_j\nu_j}u,\delta\} - \\
&F\left(x,\sum\limits_{j=1}^d\frac{\nu_j\cdot \ev_1}{\abs{\nu_j}}\Dt_{\nu_j}u,\ldots,\sum\limits_{j=1}^d\frac{\nu_j\cdot \ev_d}{\abs{\nu_j}}\Dt_{\nu_j}u\right).
\end{split}
\eq

\begin{theorem}[Convergence to Viscosity Solution]\label{thm:conv}
Let the PDE~\eqref{eq:gMA} have a unique viscosity solution and let the right-hand side $F(x,\nabla u)$ be Lipschitz continuous on $\overline{\Omega}\times\R^d$ with Lipschitz constant $K_F$.
Then the solutions of the scheme~\eqref{eq:MAmonDisc} converges to the viscosity solution of~\eqref{eq:MA} as $h,d\theta,\delta\to0$ with $\delta^{d-1}>K_F\abs{\nu_j}h/2$ for every $\nu_j\in\G$.
\end{theorem}
\begin{proof}
The convergence follows from verifying consistency and degenerate ellipticity.   This is accomplished in Lemmas~\ref{lem:deg}-\ref{lem:consistency}.
\end{proof}

\begin{lemma}\label{lem:degterms}
Under the hypotheses of \autoref{thm:conv}, the scheme for $G_{(\nu_1,\ldots,\nu_d)}^{h,d\theta,\delta}[u]$ in~\eqref{eq:monterm} is degenerate elliptic.
\end{lemma}
\begin{proof}
We introduce the notation
\[ p_j^+(x_i) = u(x_i+h\nu_j)-u(x_i),\quad p_j^-(x_i) = u(x_i-h\nu_j)-u(x_i). \]
This allows us to write $G_{(\nu_1,\ldots,\nu_d)}^{h,d\theta,\delta}[u]$ in the form of Definition~\ref{def:deg} as follows:
\begin{multline} G_{(\nu_1,\ldots,\nu_d)}^{h,d\theta,\delta}(p_1^+,p_1^-,\ldots,p_d^+,p_d^-) = 
   \prod\limits_{j=1}^d \max\left\{\frac{p_j^++p_j^-}{\abs{\nu_j}^2h^2},\delta\right\} \\+
   \sum\limits_{j=1}^d\min\left\{\frac{p_j^++p_j^-}{\abs{\nu_j}^2h^2},\delta\right\}- 
   F\left(\frac{p_1^+-p_1^-}{2\abs{\nu_1}h},\ldots,\frac{p_d^+-p_d^-}{2\abs{\nu_d}h}\right). \end{multline}
Now we need only check that this is non-decreasing in each of its arguments.  We verify this for the term $p_1^+$; the reasoning is identical for the remaining terms.

Choose any $\epsilon>0$ and consider:
\begin{align*}
G_{(\nu_1,\ldots,\nu_d)}^{h,d\theta,\delta}&(p_1^++\epsilon)-G_{(\nu_1,\ldots,\nu_d)}^{h,d\theta,\delta}(p_1^+) \\
  &\geq \delta^{d-1}\left(\max\left\{\frac{p_1^++\epsilon+p_1^-}{\abs{\nu_1}^2h^2},\delta\right\}-\max\left\{\frac{p_1^++p_1^-}{\abs{\nu_1}^2h^2},\delta\right\}\right)\\
  &\phantom{\geq} + \delta^{d-1}\left(\min\left\{\frac{p_1^++\epsilon+p_1^-}{\abs{\nu_1}^2h^2},\delta\right\}-\min\left\{\frac{p_1^++p_1^-}{\abs{\nu_1}^2h^2},\delta\right\}\right)\\
  &\phantom{\geq}-K_F\left(\frac{p_1^++\epsilon-p_1^-}{2\abs{\nu_1}h}-\frac{p_1^+-p_1^-}{2\abs{\nu_1}h}\right).
\end{align*}
In the above, we have used the facts that
\[ \min\left\{\frac{p_1^++\epsilon+p_1^-}{\abs{\nu_1}^2h^2},\delta\right\}-\min\left\{\frac{p_1^++p_1^-}{\abs{\nu_1}^2h^2},\delta\right\} \geq 0  \]
and that $\delta<1$.

We continue with this expression to conclude that
\begin{align*}
G_{(\nu_1,\ldots,\nu_d)}^{h,d\theta,\delta}&(p_1^++\epsilon)-G_{(\nu_1,\ldots,\nu_d)}^{h,d\theta,\delta}(p_1^+) \\
  &\geq \delta^{d-1}\left(\frac{p_1^++\epsilon+p_1^-}{\abs{\nu_1}^2h^2} + \delta - \frac{p_1^++p_1^-}{\abs{\nu_1}^2h^2}-\delta\right)-K_F\frac{\epsilon}{2\abs{\nu_1}h}\\
  &= \frac{\epsilon}{\abs{\nu_1}^2h^2}(\delta^{d-1}-K_F\abs{\nu_1}h/2).
\end{align*}
This expression is positive as long as $\delta^{d-1}>K_F\abs{\nu_1}h/2$.

We conclude that each of the $G_{(\nu_1,\ldots,\nu_d)}^{h,d\theta,\delta}$ is increasing in each of its arguments and is thus degenerate elliptic.
\end{proof}

\begin{lemma}\label{lem:deg}
Under the hypotheses of Theorem~\ref{thm:conv}, the scheme for $MA^{h,d\theta,\delta}_M[u]$ in~\eqref{eq:MAmonDisc} is degenerate elliptic.
\end{lemma}
\begin{proof}
This scheme is the minimum of degenerate elliptic schemes, and is therefore degenerate elliptic.
\end{proof}

\begin{lemma}\label{lem:consistency}
The scheme for $MA^{h,d\theta,\delta}_M[u]$ in~\eqref{eq:MAmonDisc} is consistent with the \MA equation~\eqref{eq:MAconv} for any function $u\in C^2(X)$.
\end{lemma}
\begin{proof}
The proof of this is identical to the consistency proof in~\cite[Lemma~6]{FOTheory}.
\end{proof}

\subsection{Hybrid discretization}\label{sec:discHybrid}
As in~\cite{FONum}, we can improve the accuracy of the discretization by using the monotone scheme only in regions of the domain where the solution may be singular.  In smooth regions of the domain, we simply use a standard centered difference discretization.

Given the regularity results described in~\S\ref{sec:reg}, we cannot expect solutions to be smooth at the boundary since the domain $X$ that we are computing on is a square, which is not strictly convex.  However, we can expect more regularity in the interior of the domain as long as the density functions are $C^\alpha$ and are bounded away from 0 and $\infty$.

We first identify $X^s$, which is a neighborhood of any regions where $u$ may be singular:
\[
X^s = 
\partial X \cup
\{ x \in X \mid  f(x)<\epsilon  \}  
\cup
\{ x \in X \mid  f(x)>1/\epsilon  \} 
\cup  
\{ x \in  X \mid f \notin C^\alpha(B(x,\epsilon)) \}.
\]
Here $\epsilon$ is a small parameter, which we can take equal to the spatial resolution $h$.

We define $w(x)$ to be a function that is one in an $h$-neighborhood of $X^s$ and that goes to zero elsewhere.  Then a possible hybrid discretization of the \MA equation is
\bq\label{eq:discHybrid}
MA^{h,d\theta,\delta}[u] = w(x)MA_M^{h,d\theta,\delta}[u] + (1-w(x)) MA_S^h[u].
\eq

\section{Numerical implementation}\label{sec:implementation}
Now that we have described the discretization we will be using, we turn our attention to the remaining details of the numerical implementation of the iteration described in this paper.  In this section, we describe our methods for enforcing boundary conditions, solving the discrete system of equations, and initializing the iteration.  To be concrete, we will describe these issues in the two-dimensional case, but the ideas easily generalize to higher dimensions.

\subsection{Existence and uniqueness}\label{sec:discExist}
We recall here that the iteration~\eqref{eq:it2} requires us to solve not only for the function $u$, but also for the scaling factor $c$ that multiplies the density functions. We can simply include this as an additional variable in the discrete system of equations.

With the addition of an extra variable, we should expect that we will also require an additional equation.  This is reasonable since solutions of the Neumann problem are only unique up to an additive constant.  We simply add an extra equation that forces $u$ to be zero at one corner of the domain. 

\subsection{Boundary conditions}\label{sec:discBC}
We also need to discretize the Neumann boundary conditions.

Our computational domain is the square, which means we must impose values for $u_{x_1}$ on the left and right sides of the domain and for $u_{x_2}$ on the top and bottom edges of the domain.

We accomplish this by adding a layer of ghost points around the outside of our computational domain.  The value of the normal derivatives on the boundary can then be discretized using simple centered differences.  For example, at a point on the left edge ($x_1=x_{\min}$), we can discretize the normal derivative as
\[ u_\nv(x) = \frac{1}{2h}(u(x_{\min}+h,x_2)-u(x_{\min}-h,x_2)). \]
The use of ghost points ensures that all values needed in this discretization are available.

We also need to provide four more equations at the corner points in our grid.  We specify the value of the derivative in the ``diagonal'' direction ($(1,1)$, $(1,-1)$, $(-1,1)$, or $(-1,-1)$) that points outward from the grid at each of these four points.  This is enforced using centered differences.  So, for example, at the points $(x_{1,\min},x_{2,\min})$ we require that
\begin{multline*} \frac{1}{2\sqrt{2}h}(u(x_{1,\min}-h,x_{2,\min}-h) - (x_{1,\min}+h,x_{2,\min}+h)) =\\ -\frac{1}{\sqrt{2}}(u_{x_1}(x_{1,\min},x_{2,\min}) + u_{x_2}(x_{1,\min},x_{2,\min})). \end{multline*}
As before, the ghost points ensure that all of these values are available.

\subsection{Newton's method}\label{sec:newton}
The discretization of the \MA equation described in~\S\ref{sec:discretize} results in a system of equations that can be solved efficiently using Newton's method.

This involves performing the iteration
\[ u^{k+1} = u^k - v^k,\quad c^{k+1} = c^k - d^k \]
where the correctors $v^k$, $d^k$ are obtained by solving the equation
\[ \nabla MA[u^k,c^k](v^k,d^k)^T = MA[u^k,c^k]. \]
As long as the initial iterate $u^0$ satisfies the given Neumann boundary condition, we can simply enforce a homogeneous Neumann condition on the corrector $v^k$ at each step.

Although we are using a hybrid discretization, the weight function that determines the discretization is independent of the iterates $u^k$, $c^k$.  This means that we can compute the Jacobians of the monotone and standard discretizations and obtain the Jacobian of the hybrid system via
\[ \nabla MA[u,c] = w(x) \nabla MA_M[u,c] + (1-w(x)) \nabla MA_S[u,c]. \]

We begin by computing the Jacobian of the monotone discretization.  We recall that this discretization has the form
\[ MA_M[u,c] = \min\limits_{(\nu_1,\ldots,\nu_d)\in\G} G_{(\nu_1,\ldots,\nu_d)}[u,c]. \]
By Danskin's Theorem~\cite{Bertsekas}, we can write the Jacobian of this as
\[ \nabla MA_M[u,c] = \nabla G_{(\nu_1,\ldots,\nu_d)}[u,c],\]
where the $(\nu_1,\ldots,\nu_d)$ are the directions active in the minimum.

This Jacobian can be broken down into two basic components: the gradient with respect to the solution vector $u$ and the gradient with respect to the scaling factor $c$.  The first component is given by:
\begin{multline*} 
\nabla_{u_i} G_{(\nu_1,\ldots,\nu_d)}[u,c] = \sum\limits_{m=1}^d\left[\left(\prod\limits_{j\neq m}\max\{\Dt_{\nu_j\nu_j}u_i,\delta\}\right)\one_{\Dt_{\nu_j\nu_j}u_i\geq\delta} + 
\one_{\Dt_{\nu_j\nu_j}u_i<\delta}\right]\Dt_{\nu_m\nu_m}\\
 - c\sum\limits_{m=1}^d\frac{\partial F}{\partial p_m}\left(x,\sum\limits_{j=1}^d\frac{\nu_j\cdot \ev_1}{\abs{\nu_j}}\Dt_{\nu_j}u_i,\ldots,\sum\limits_{j=1}^d\frac{\nu_j\cdot \ev_d}{\abs{\nu_j}}\Dt_{\nu_j}u_i\right)
 \sum\limits_{j=1}^d\frac{\nu_j\cdot\ev_m}{\abs{\nu_j}}\Dt_{\nu_j}.
\end{multline*}
The final component is given by
\[ \nabla_c G_{(\nu_1,\ldots,\nu_d)}[u,c] = -F\left(x,\sum\limits_{j=1}^d\frac{\nu_j\cdot \ev_1}{\abs{\nu_j}}\Dt_{\nu_j}u_i,\ldots,\sum\limits_{j=1}^d\frac{\nu_j\cdot \ev_d}{\abs{\nu_j}}\Dt_{\nu_j}u_i\right).\]

We also require the Jacobian of the standard discretization.  In this case, the first component of the Jacobian (in two dimensions) is simply
\begin{multline*}  
\nabla_{u_i} MA_S[u,c] = (\Dt_{x_2x_2}u_i)\Dt_{x_1x_1} + (\Dt_{x_1x_1}u_i)\Dt_{x_2x_2} + 2(\Dt_{x_1x_2}u_i)\Dt_{x_1x_2} \\ 
-c\frac{\partial F}{\partial p_1}(x,\Dt_{x_1} u_i,\Dt_{x_2}u_i)\Dt_{x_1} - c\frac{\partial F}{\partial p_2}(x,\Dt_{x_1} u_i,\Dt_{x_2}u_i)\Dt_{x_2}
\end{multline*}
and the second component is
\[ \nabla_c MA_S[u,c] = -F(x,\Dt_{x_1} u_i,\Dt_{x_2}u_i). \]

\subsection{Initialization}\label{sec:init}
The iterations we have described in this paper also need to be initialized.  There are really two aspects to this: we need to initialize $u$ and $c$ each time we solve the \MA equation and we also need to initialize our estimation of the boundary conditions $\phi(x)$.

\subsubsection{Initialization of boundary data}
First we discuss the initialization of the boundary data $\phi^0$ in the iteration~\eqref{eq:it2}.  The simplest approach would be to extract boundary conditions from the identity map $s(x) = x$.  However, if this mapping does not overlap with the target set $Y$, the iteration is likely to fail.  

We can remedy this problem by instead extracting boundary data from the scaled identity map $s(x) = Mx$ where the constant M is chosen large enough so that the set $s(X)$ encompasses the target set $Y$.

Once this constant is chosen, we simply choose the initial boundary condition
\[ \phi^0(x) = Mx\cdot\nv(x),\quad x\in\partial X. \]

We can accelerate the convergence of this method by first solving the transport problem on a coarser grid, then interpolating the resulting boundary data onto the refined mesh.

\subsubsection{Initialization of Newton's method}
We also need to initialize Newton's method each time we solve the \MA equation.
We can use the approach described in~\cite{FOTheory}, which involves obtaining the initial guess by solving the equation
\[
\Delta u(x) = (c{d!f(x)/g(x-x_0)})^{1/d}
\]
where $x_0$ is a point in the interior of the target set $Y$.

Since we will be solving the \MA equation multiple times with different boundary conditions, we can also accelerate the convergence of the $(k+1)^{st}$ iteration by initializing with the solution found during the previous solve ($u^k$).  One important point here is that the boundary data changes from step to step.  Thus it is important to change the values of $u^k$ at the boundary points so as to ensure that correct boundary conditions are satisfied.

\section{Computational results: the \MA equation}\label{sec:resultsMA}
In this section, we provide computational results for several different examples.  We have really introduced two ideas in this paper: a discretization for \MA type equations and a method for implementing the transport boundary  condition.  In this section, we focus on testing our \MA solver.  To keep this idea clear, we restrict ourselves to the problem of mapping rectangles to rectangles; in this case, our method for the transport problem is reduced to a single \MA solve with Neumann boundary conditions (see~\S\ref{sec:neumann}).

In each example our domain is a square, which is discretized on an $N\times N$ grid using a 17 point stencil.  We let $h = 1/(N-1)$ denote the spatial resolution of the grid and let $M = N^2$ denote the total number of grid points.  
The computations were done in MATLAB on a laptop with a 2 GHz Intel processor.  

When an exact solution $u^{exact}$ is available, we provide the maximum error in the gradient map:
\[ \text{Error} = \max\{\|u^{exact}_{x_1}-u_{x_1}\|_\infty,\|u^{exact}_{x_2}-u_{x_2}\|_\infty\}. \]
We also provide the total number of Newton iterations and computation time required for each example.

The examples we consider include:
\begin{itemize}
\item A (linear) map between gaussian densities.
\item A comparison between a map obtained by solving the direct problem and a map obtained by inverting the solution to the inverse problem.
\item A map from a uniform density onto a density that blows up at a point.
\item A map between two brain MRI images.
\end{itemize}

\subsection{Gaussian densities}\label{sec:gaussian}
We begin by showing that we can recover a linear mapping between two rectangles with gaussian densities.  We consider the problem of mapping the square $(-0.5,0.5)\times(-0.5,0.5)$ onto the rectangle $(0.5,1.5)\times(-1,1)$ with the density functions:
\[ f(x_1,x_2) = \frac{1}{0.16}\exp{\left(-\frac{1}{2}\frac{x_1^2}{0.4^2}-\frac{1}{2}\frac{x_2^2}{0.4^2}\right)},\] 
\[g(y_1,y_2) = 
  \frac{1}{0.08}\exp{\left(-\frac{1}{2}\frac{(y_1-1)^2}{0.4^2}-\frac{1}{2}\frac{y_2^2}{0.2^2}\right)}. \]
In this case, we have an explicit expression for the optimal map:
\[ u_{x_1} = x_1 + 1,\quad u_{x_2} = \frac{1}{2}x_2. \]

We present the results in \autoref{table:gaussian} and \autoref{fig:gaussian}.  
In this example, we can actually achieve machine accuracy (if we take enough Newton steps).    This is because the exact solution is simply a linear map, which will exactly solve the discretized system of equations.
In addition to this, we find that the Newton solver for the \MA equation converges in $\bO(M)$ time.

\begin{figure}[htdp]
	\centering
%	\subfigure[]{\includegraphics[width=.48\textwidth]{squareX}\label{fig:squareX}}
%        \subfigure[]{\includegraphics[width=.48\textwidth]{gaussY}\label{fig:gaussY}}
	\subfigure[]{\includegraphics[width=.4\textwidth]{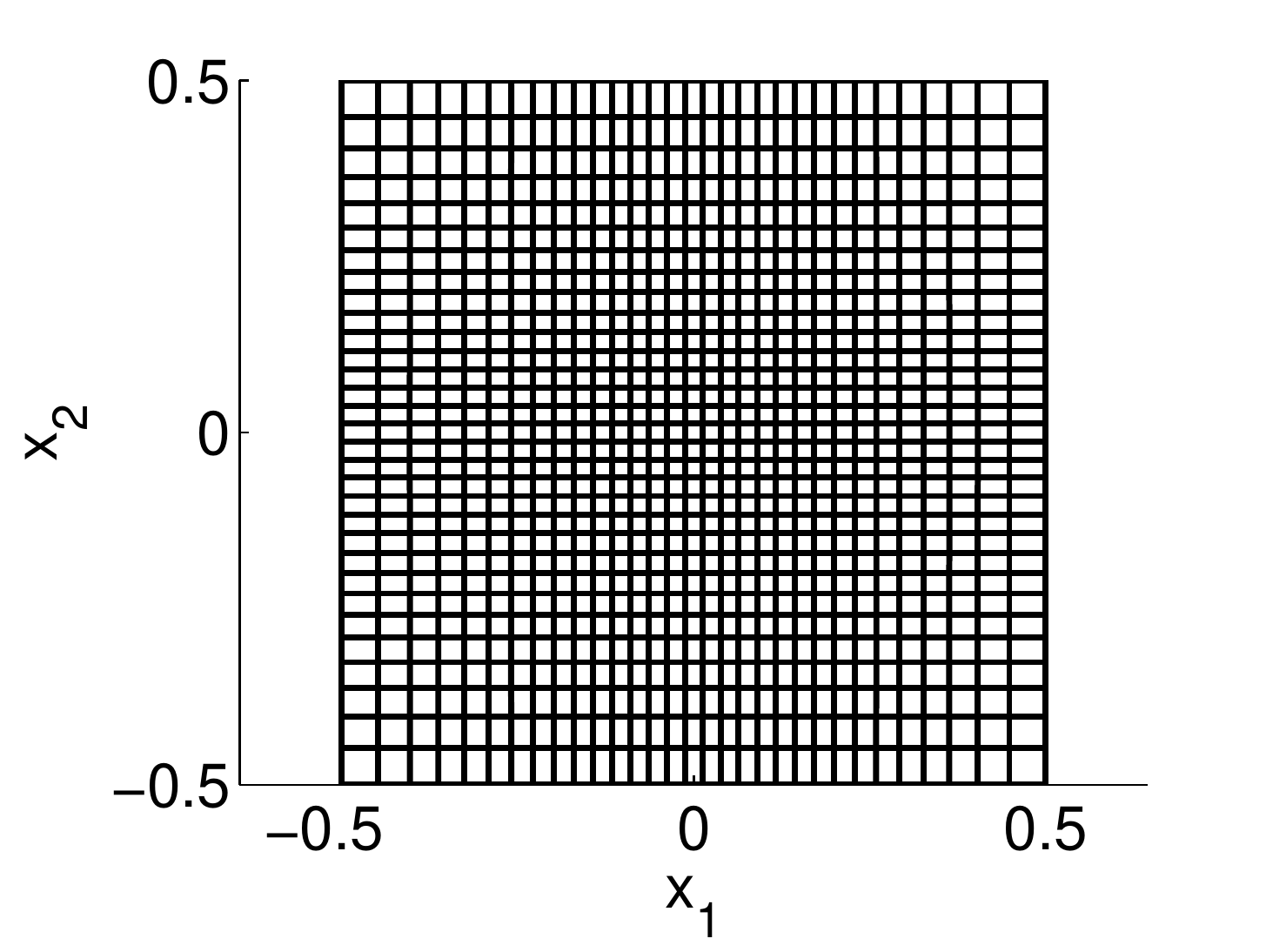}\label{fig:gaussianX}}
	\subfigure[]{\includegraphics[width=.4\textwidth]{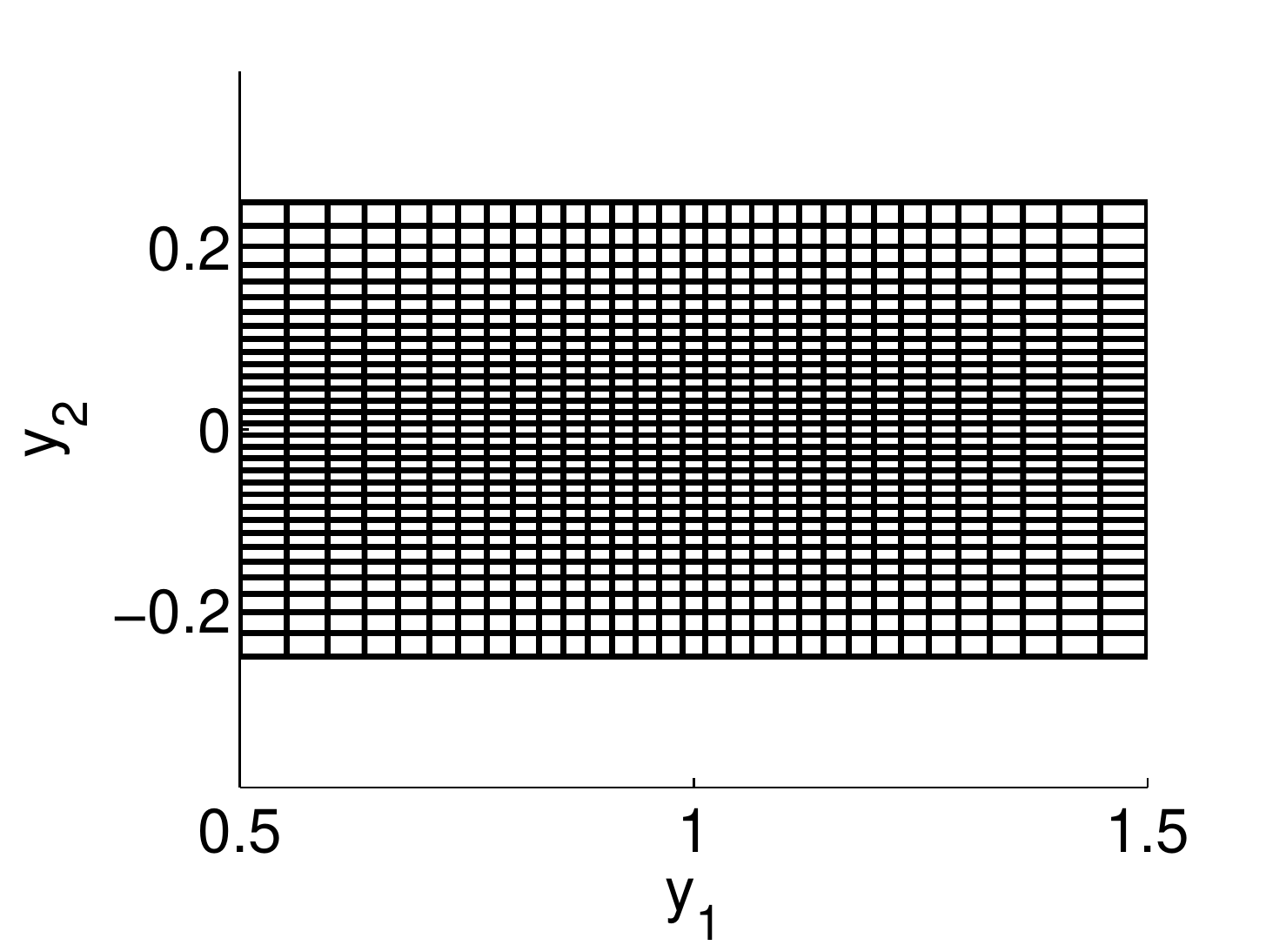}\label{fig:gaussianY}}
  	\vspace*{-12pt}\caption{
  	\subref{fig:gaussianX} A mesh with gaussian density $f$ and \subref{fig:gaussianY} its image under the gradient map $\nabla 	u$~(\S\ref{sec:gaussian}). }
  	\label{fig:gaussian}
\end{figure} 

\begin{table}[htdp]\small
\begin{center}
\begin{tabular}{ccccc}
N  & h & Newton Iterations & CPU Time (s) & {Maximum Error} \\
\hline
32 & 0.0323 & 1 & 0.1 & $5.71\times10^{-8}$ \\
46 & 0.0222 & 1 & 0.2 & $3.34\times10^{-8}$ \\
64 & 0.0159 & 1 & 0.3 & $0.26\times10^{-8}$ \\
90 & 0.0112 & 1 & 0.6 & $0.18\times10^{-8}$ \\
128 & 0.0079 & 1 & 1.1 & $0.13\times10^{-8}$\\
182 & 0.0055 & 1 & 2.4 & $0.09\times10^{-8}$\\
256 & 0.0039 & 1 & 5.3 & $0.07\times10^{-8}$\\
362 & 0.0028 & 1 & 12.4 & $0.05\times10^{-8}$
\end{tabular}
\end{center}
\caption{Computation time and maximum error for the map between two gaussian densities (\S\ref{sec:gaussian}).}
\label{table:gaussian}
\end{table}

\subsection{Recovering an inverse map}\label{sec:inverse}
For our next example, we consider another problem with an exact solution, which will be used to verify that we can correctly recover inverse maps.  To set up this example, we define the function
\[ q(z) = \left(-\frac{1}{8\pi}z^2 + \frac{1}{256\pi^3}+\frac{1}{32\pi}\right)\cos(8\pi z) + \frac{1}{32\pi^2}z\sin(8\pi z). \]
Now we map the density 
\[ f(x_1,x_2) = 1+4(q''(x_1)q(x_2)+q(x_1)q''(x_2)) + 16(q(x_1)q(x_2)q''(x_1)q''(x_2)-q'(x_1)^2q'(x_2)^2) \]
in the square $(-0.5,0.5)\times(-0.5,0.5)$ onto a uniform density in the same square.  This transport problem has the exact solution
\[ u_{x_1}(x_1,x_2) = x_1 + 4q'(x_1)q(x_2),\quad u_{x_2}(x_1,x_2) = x_2 + 4q(x_1)q'(x_2). \]

We will solve this problem in two ways:
\begin{itemize}
\item Directly, as in the previous example.
\item By solving the inverse problem (mapping $g$ to $f$) and inverting the resulting map.
\end{itemize}

Results are presented in \autoref{fig:inverse} and \autoref{table:inverse}.  We find that the maps obtained from both the forward and inverse formulations have about $\bO(h^2)$ accuracy.  Both problems are solved in about $\bO(M)$ time.

\begin{figure}[htdp]
	\centering
	\subfigure[]{\includegraphics[width=.4\textwidth]{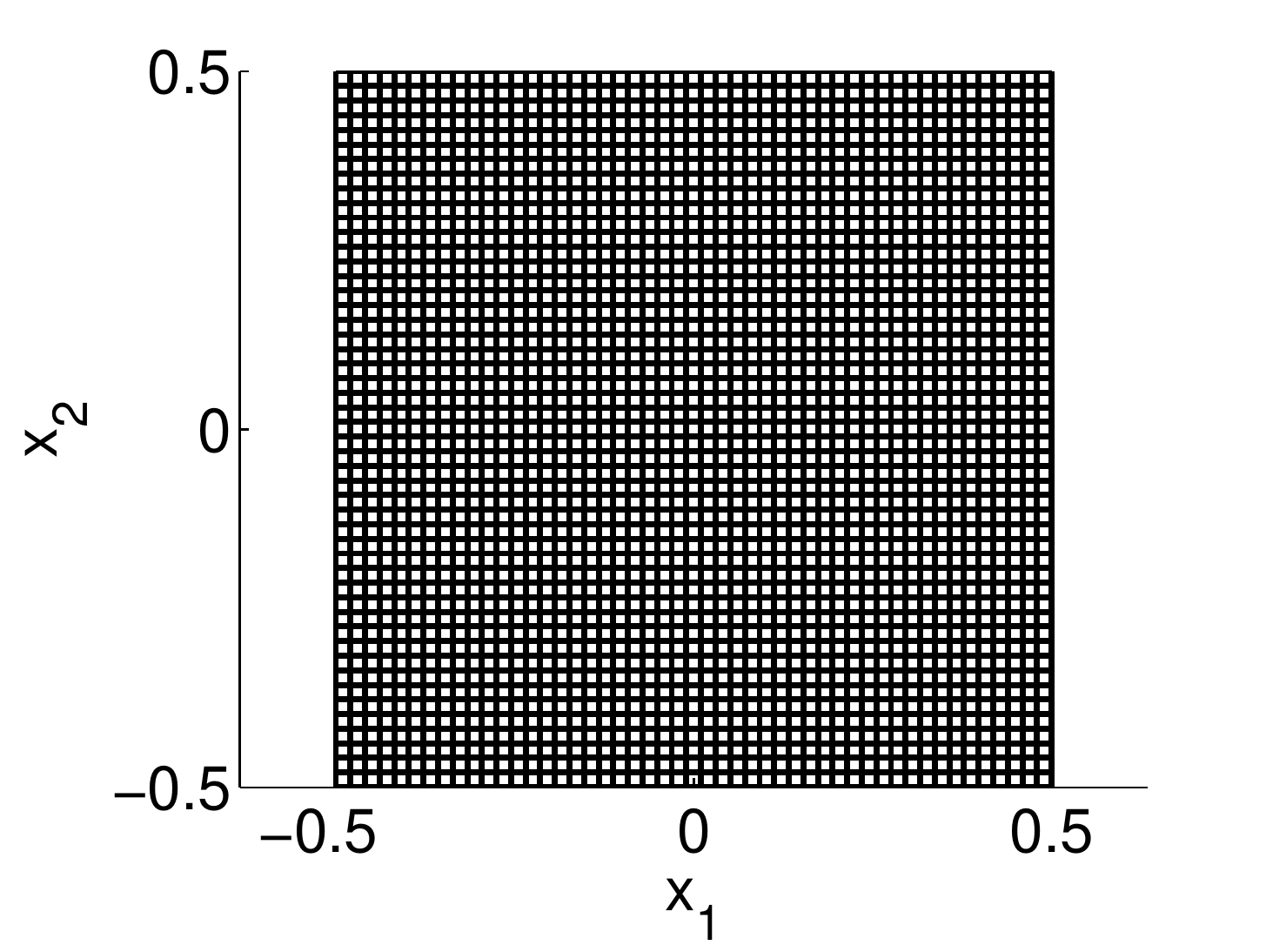}\label{fig:invX}}
	\subfigure[]{\includegraphics[width=.4\textwidth]{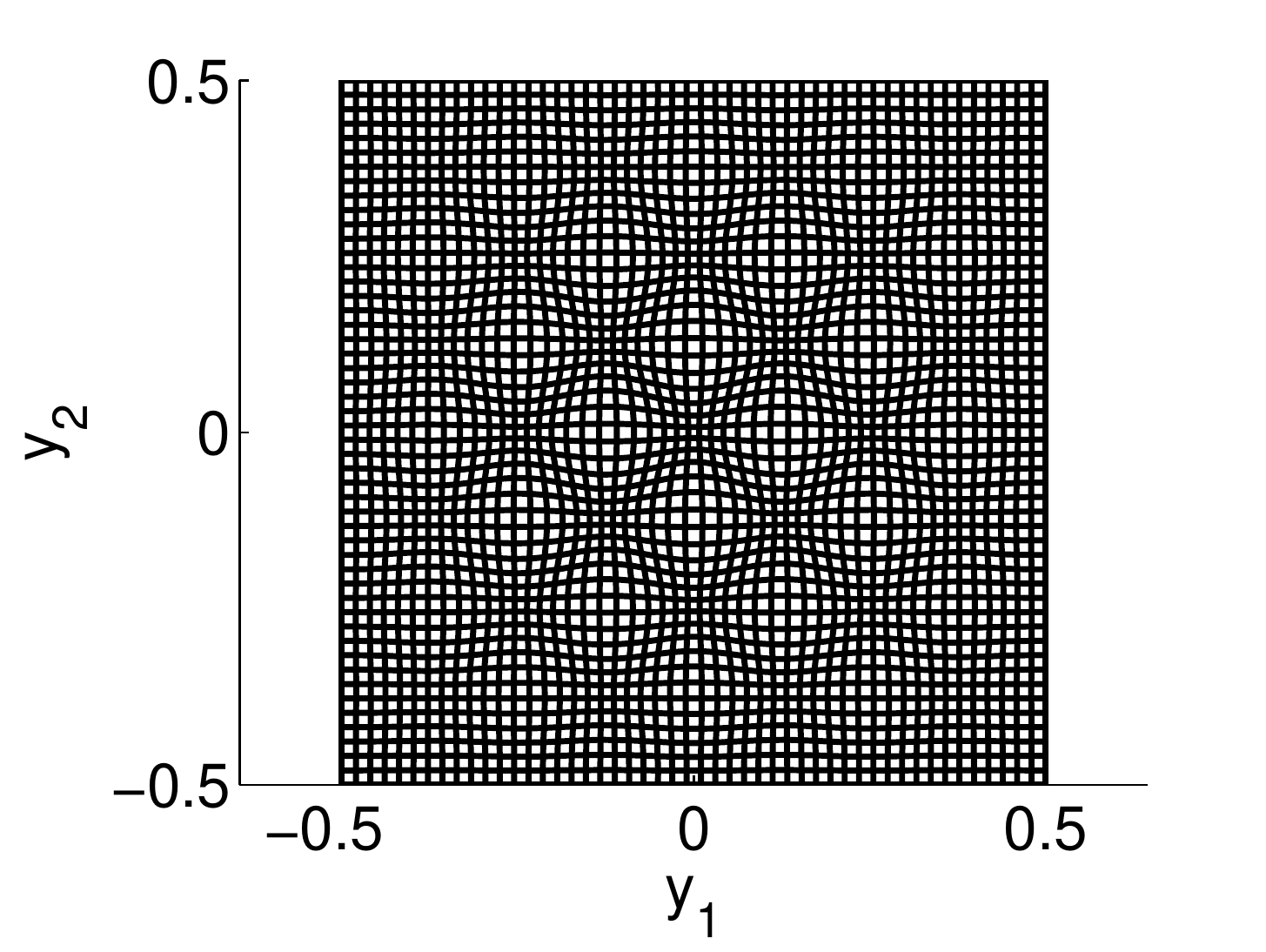}\label{fig:invY}}
  	\vspace*{-12pt}\caption{
  	\subref{fig:invX} A uniform cartesian mesh and \subref{fig:invY} its image under the gradient map $\nabla 	u$ (\S\ref{sec:inverse}). }
  	\label{fig:inverse}
\end{figure} 

\begin{table}[htdp]\small
\begin{center}
\begin{tabular}{ccccccc}
 &  \multicolumn{3}{c}{Forward Problem} & \multicolumn{3}{c}{Inverse Problem} \\
N  &  Iterations & Time (s) & {Max Error} & Iterations & Time (s) & {Max Error}\\
\hline
32 & 3 & 0.2 & $2.476\times10^{-3}$ & 4 & 0.4 & $2.450\times10^3$\\
46 & 2 &  0.2 & $0.631\times10^{-3}$ & 2 & 0.5 & $0.575\times10^3$\\\
64 & 2& 0.5 &   $0.241\times10^{-3}$ & 2& 1.1 & $0.244\times10^3$\\\
90 &  1 & 0.6 & $0.106\times10^{-3}$ & 1& 1.3 & $0.101\times10^3$\\\
128 & 1 & 1.3 & $0.049\times10^{-3}$ & 1& 2.9 & $0.048\times10^3$\\\
182 & 1 & 2.9 & $0.024\times10^{-3}$ & 1 & 5.1 &$0.023\times10^3$\ \\
256 & 1 & 6.3 & $0.012\times10^{-3}$ & 1& 10.9 & $0.011\times10^3$\\\
362 &  1 & 14.0 & $0.006\times10^{-3}$ & 1 & 22.6 &$0.006\times10^3$\
\end{tabular}
\end{center}
\caption{Newton iterations, computation time and maximum error for a map obtained by a direct solve and by inverting the inverse map (\S\ref{sec:inverse}).}
\label{table:inverse}
\end{table}

\subsection{An example with blow-up}\label{sec:exBlowup}
Next we consider the problem of mapping a uniform density onto a density that blows up at a point:
\[ g(y_1,y_2) =  \frac{\exp\left(-2\sqrt{(y_1-0.5)^2+(y_2-0.5)^2}\right)}{\sqrt{(y_1-0.7)^2+(y_2-0.7)^2}}.\]
In this case, both $X$ and $Y$ are the square $(0,1)\times(0,1)$.  This example is taken from~\cite{Delzanno}, which allows us to compare results.  In this example, we slightly regularize the density $g$ (bounding it by a $\bO(1/h^2)$ function) to prevent infinities from appearing.

We present the timing results in Table~\ref{table:blowup}.  We provide not only the number of Newton iterations and computation time, but also the ratio
\[ R = \max\left\{g(y_1,y_2)/f(x_1,x_2)\right\}, \]
since many currently available \MA solvers can become slow or unstable when this ratio is large.
For comparison, we provide the same information for the method of~\cite{Delzanno} (which is essentially our ``standard'' discretization solved with an optimized Newton-Krylov method).  The method of~\cite{Delzanno} runs in $\bO(M)$ time. Our method, though it runs in about $\bO(M^{1.1})$ time, has lower computation times and deals with larger density ratios.  Naturally, we cannot conclude too much from the comparison of computation times since the computations were performed on different computers.  However, it is evident that, in terms of computation time, our method is very competitive with other fast solvers.

We also present the deformed mesh and zoom into the region of high density to verify that our method has produced an untangled mesh; see Figure~\ref{fig:blowup}.

\begin{table}[htdp]\small
\begin{center}
\begin{tabular}{ccccccc}
 & \multicolumn{3}{c}{Hybrid Method} & \multicolumn{3}{c}{Method of~\cite{Delzanno}}\\
N   & R & Iterations & CPU Time (s) & R & Iterations & CPU Time (s) \\
\hline
32  & 546 & 4& 0.2 & 356 & 6 & 1  \\
46  &  1,151 & 4 & 0.3 & --- & --- & --- \\
64  &  2,254 & 5 & 0.8 & 1,127 & 7 & 4 \\
90  &  4,066 & 5 & 1.6 & --- & --- & ---\\
128  & 9,162 & 5 & 3.5 & 2,829 & 7 & 17.4\\
182  & 18,608 & 5 & 8.3 & --- & --- & ---\\
256  & 36,933 & 5 & 19.4 & 8,886 & 7 & 70\\
362  & 74,018 & 4 & 36.3 & --- & --- & ---
\end{tabular}
\end{center}
\caption{Ratio of density functions, Newton iterations, and total computation time for the hybrid method~(\S\ref{sec:discHybrid},\ref{sec:newton}) and the method of~\cite{Delzanno}.}
\label{table:blowup}
\end{table}

\begin{figure}[htdp]
	\centering
	\subfigure[]{\includegraphics[width=.4\textwidth]{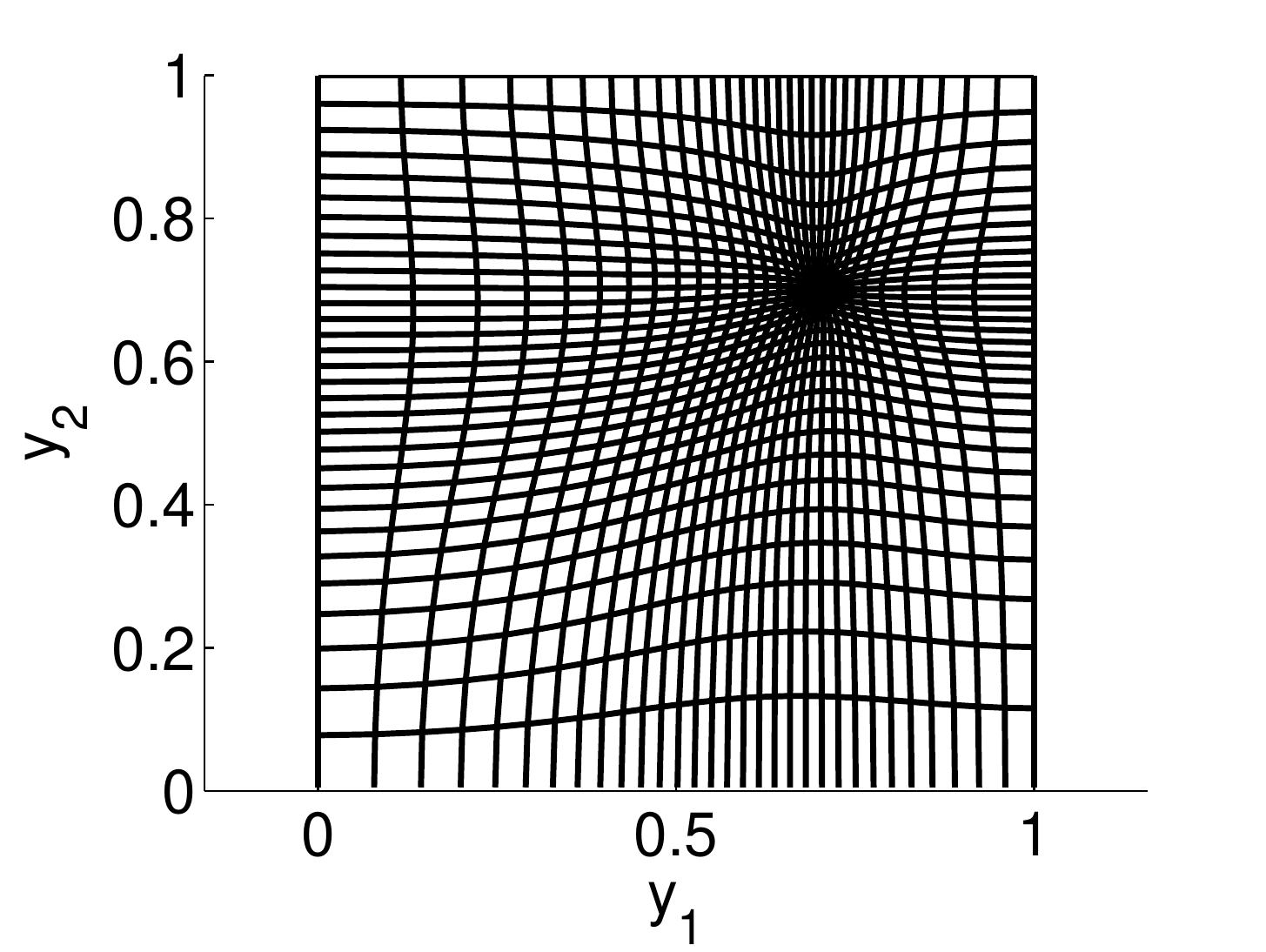}\label{fig:blowupMesh}}
	\subfigure[]{\includegraphics[width=.4\textwidth]{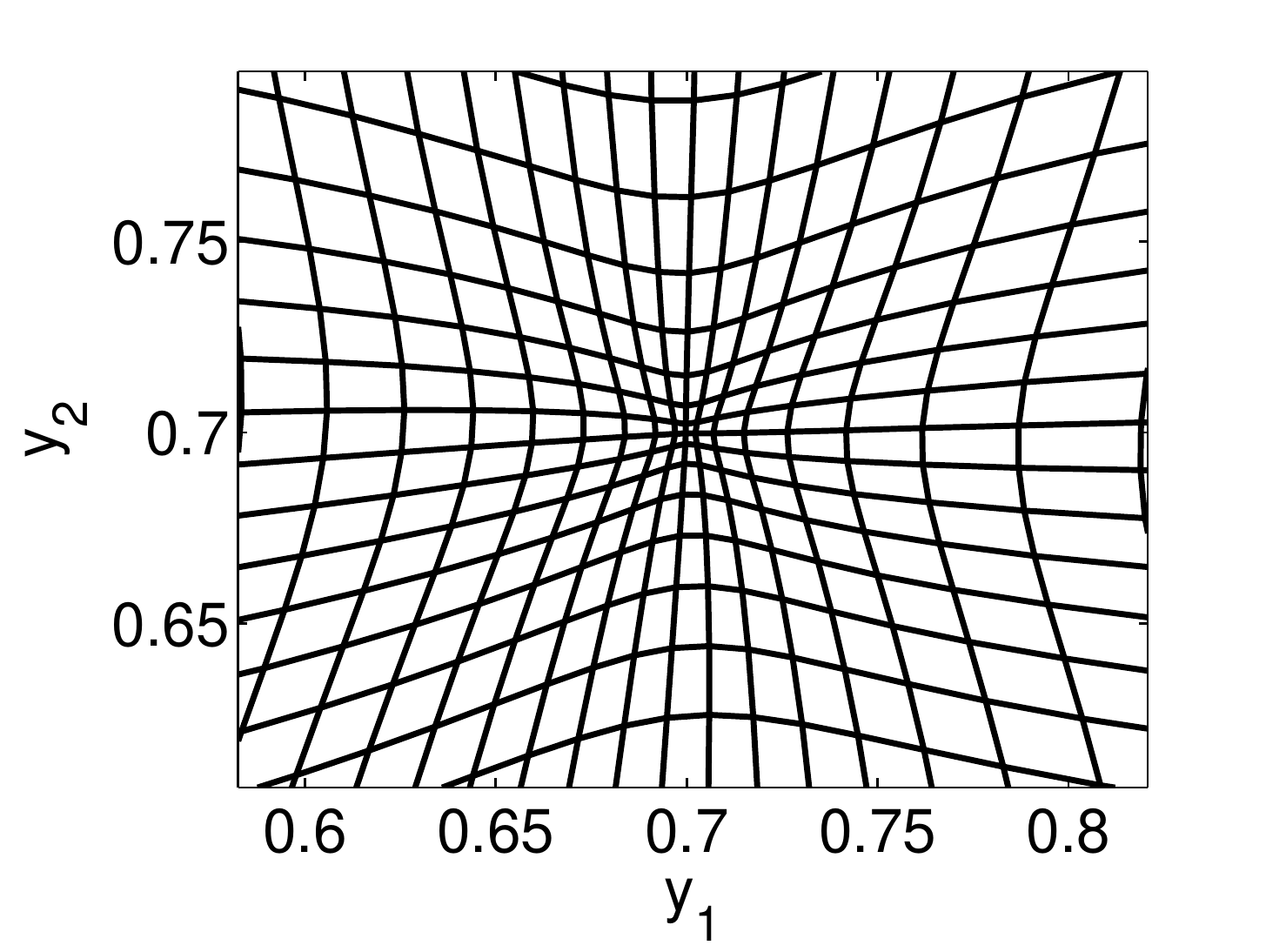}\label{fig:blowupZoom}}
  	\vspace*{-12pt}\caption{
  	\subref{fig:blowupMesh} The image of a cartesian mesh under the gradient map $\nabla 	u$ (\S\ref{sec:exBlowup}) and \subref{fig:blowupZoom} a zoomed in view of the same mesh in the region of large density. }
  	\label{fig:blowup}
\end{figure}

\subsection{Mapping between brain MRI images}\label{sec:brains}
We conclude this section with an example from image processing.  In this example, we obtain our density functions from the pixel intensities in two synthetic brain MRI images~\cite{brainweb,CollinsBrains,CocoscoBrains}.  The images are shown in Figures~\ref{fig:brain1}-\ref{fig:brain2}.  In this case, the regions $X$ and $Y$ are identical and are equal to the unit square.  The fully resolved images contain $256\times256$ pixels.  For the computations presented here, we have also interpolated both images onto coarser grids so that in each case we are mapping an $N\times N$ grid onto the density function obtained from an $N\times N$ image.  

In this example, the density functions have large gradients, which effectively increase as we map onto more refined images.  The solver now runs in about $\bO(M^{1.1})$ time; see~\autoref{table:brains}.

%Figures~\ref{fig:brainf2}-\ref{fig:braing2} show the mesh with density given by the first image (this is obtained by mapping a uniform mesh onto the first brain) and its image under the gradient map.  While the thickness of the grid lines obscures some of the fine details in the brain images, it is evident that the density of the grid lines is mimicing the pixel intensity in the images.  
Figures~\ref{fig:brainMapped}-\ref{fig:brainErr} show the image we obtain by solving the \MA equation and interpolating and the error in this image.   The mapped image we obtain agrees well with the given image.  Not surprisingly, the largest error occurs around the edges of the brain where the density function is essentially discontinuous; consequently, small errors in the map can lead to large errors in estimated pixel intensity.

\begin{figure}[htdp]
	\centering
	\subfigure[]{\includegraphics[width=.3\textwidth]{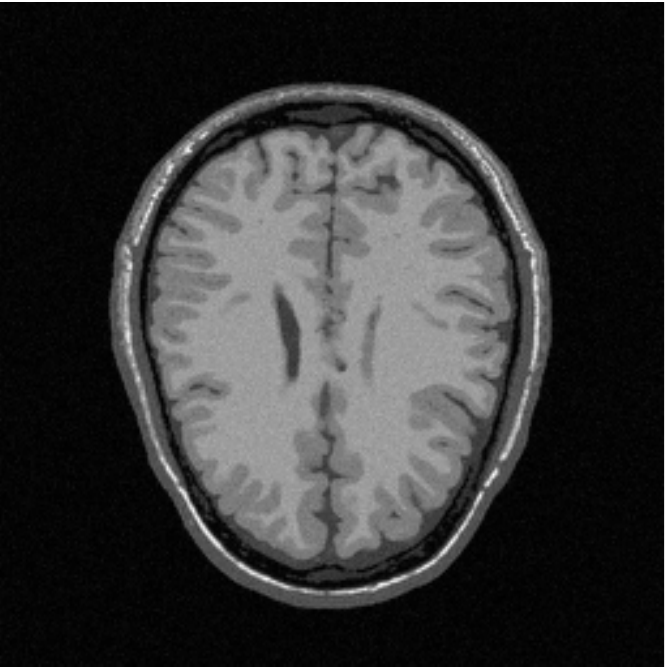}\label{fig:brain1}}
        \subfigure[]{\includegraphics[width=.3\textwidth]{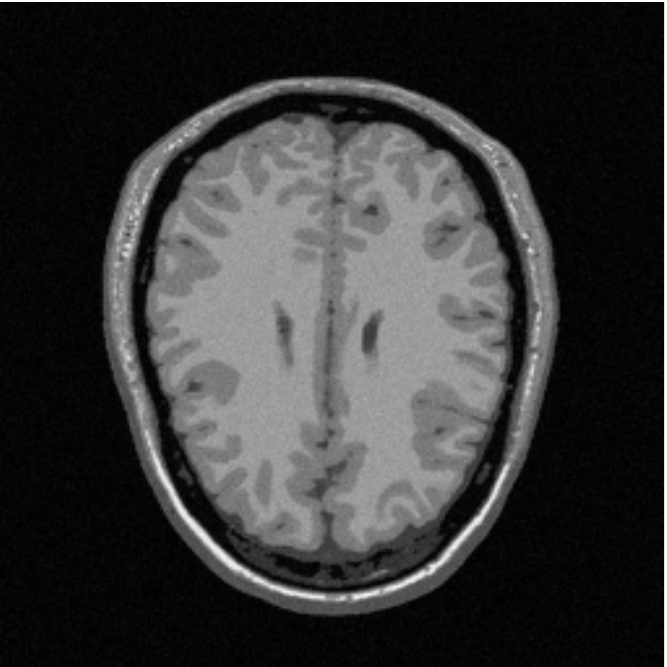}\label{fig:brain2}} \\
%        \subfigure[]{\includegraphics[width=.4\textwidth]{brainf2}\label{fig:brainf2}}
%        \subfigure[]{\includegraphics[width=.4\textwidth]{braing2}\label{fig:braing2}}
        \subfigure[]{\includegraphics[width=.3\textwidth]{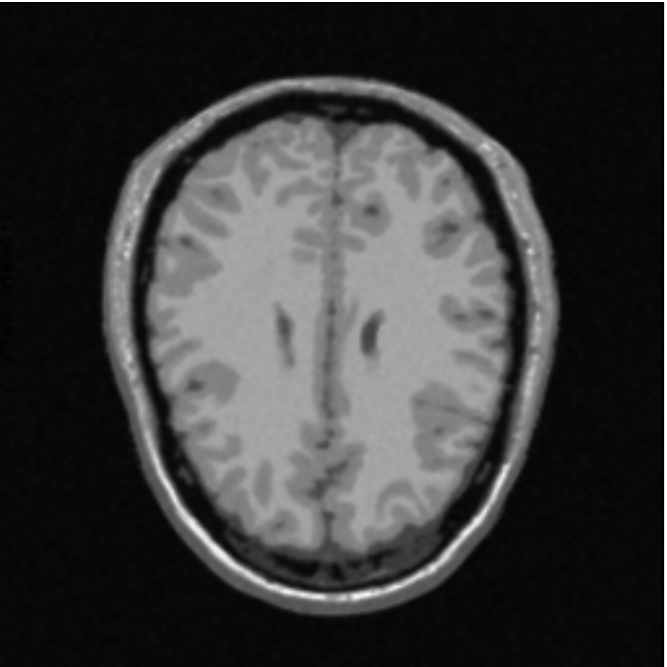}\label{fig:brainMapped}}
        \subfigure[]{\includegraphics[width=.3\textwidth]{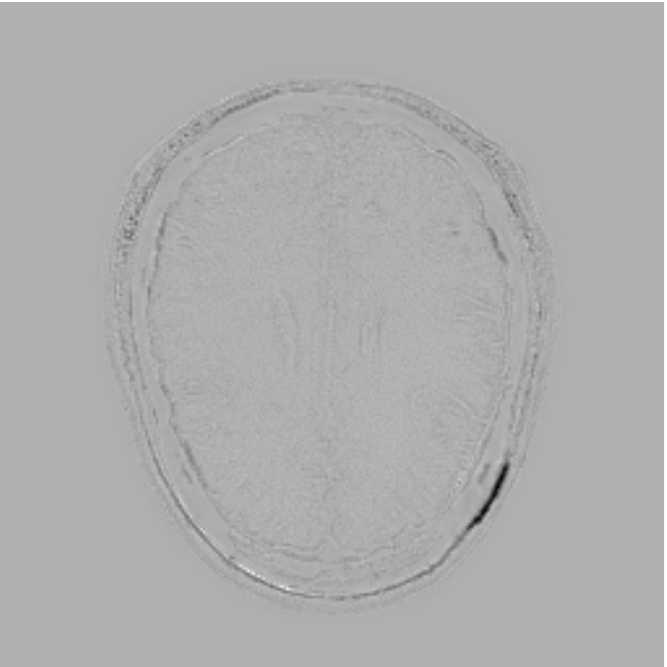}\label{fig:brainErr}}
  	\vspace*{-12pt}\caption{\subref{fig:brain1} The initial density function $f$, \subref{fig:brain2} the final density function $g$,  \subref{fig:brainMapped} the image obtained by solving the \MA equation and interpolating, and \subref{fig:brainErr} the error in the resulting image.}
  	\label{fig:brains}  	
\end{figure}

%\begin{figure}[htdp]
%	\centering
%        \subfigure[]{\includegraphics[width=\textwidth]{brainf2}\label{fig:brainf2}}
%        \subfigure[]{\includegraphics[width=\textwidth]{braing2}\label{fig:braing2}}
%  	\caption{\subref{fig:brainf2} A mesh with density function given by the first image, \subref{fig:braing2} its image under the gradient map $\nabla 	u$ (\S\ref{sec:brains})}
%  	\label{fig:brains}  	
%\end{figure} 

\begin{table}[htdp]\small
\begin{center}
\begin{tabular}{ccc}
N   & Newton Iterations& CPU Time (s)  \\
\hline
32 & 7 & 1.1\\
46 & 7 & 1.2\\
64 & 9 & 3.0\\
90 & 10 & 7.0\\
128 & 12 & 13.7\\
182 & 12 & 34.9\\
256 & 13 & 81.6
\end{tabular}
\end{center}
\caption{Computation time for a map between two brain MRI images (\S\ref{sec:brains}).}
\label{table:brains}
\end{table}

\section{Computational results: optimal transport}\label{sec:resultsTransport}
In this section, we turn our attention to computational results for the transport problem. In each example, we embed our domain in the square $(-0.5,0.5)\times(-0.5,0.5)$ (setting the density $f=0$ outside our domain $X$).  While this can lead to singularities in the solutions, our methods are robust enough to handle this non-smoothness.

In each case, we present the total number of \MA solves required on the $N\times N$ grid (this does not include solves performed on coarser grids during the initialization process), as well as the total computation time required.  When an exact solution is available for comparison, we provide the maximum error in the map:
\[ \text{Error} = \max\{\|u^{exact}_{x_1}-u_{x_1}\|_\infty,\|u^{exact}_{x_2}-u_{x_2}\|_\infty\}. \]

The examples considered in this section include:
\begin{itemize}
\item A map between two ellipses, for which an exact solution is available for comparison.
\item A map from two disconnected semi-circles onto a circle, for which an exact solution is available for comparison.
\item A map from a square onto a convex polygon, which is neither smooth nor strictly convex, together with recovery of the inverse map.
\item A map from a square onto a non-convex region.
\end{itemize}

\subsection{Mapping an ellipse to an ellipse}\label{sec:exEllipse}
First we consider the problem of mapping an ellipse onto an ellipse.  
To describe the ellipses, we let $M_x,M_y$ be symmetric positive definite matrices and let $B_1$ be the unit ball in $\R^d$.
Now we take $X = M_xB_1$, $Y = M_yB_2$ to be ellipses with constant densities $f$, $g$ in each ellipse.

In $\R^2$, the optimal map can be obtained explicitly~\cite{MOEllipse} from
\[ \nabla u(x) = M_yR_\theta M_x^{-1}x\]
where $R_\theta$ is the rotation matrix
\[ R_\theta = \left(\begin{array}{cc} \cos(\theta) & -\sin(\theta)\\ \sin(\theta) & \cos(\theta)\end{array}\right), \]
the angle $\theta$ is given by
\[ \tan(\theta) = \trace(M_x^{-1}M_y^{-1}J)/\trace(M_x^{-1}M_y^{-1}), \]
and the matrix $J$ is equal to
\[ J = R_{\pi/2} = \left(\begin{array}{cc} 0 & -1\\ 1 & 0\end{array}\right). \]

We use the particular example
\[ M_x = \left(\begin{array}{cc}0.4 & 0\\0 & 0.2 \end{array}\right) , 
\quad M_y = \left(\begin{array}{cc} 0.3 & 0.1\\0.2 & 0.4\end{array}\right),\]
which is pictured in \autoref{fig:ellipses}.

Projections onto the ellipse at each step are accomplished efficiently using the method described in~\cite{KiselevProject}.

Computational results are presented in \autoref{table:ellipses} and \autoref{fig:ellipses}.  The error is decreasing uniformly (about $\bO(h^{0.8})$).  We cannot expect high accuracy for this example due to its degeneracy: the density $f$ vanishes in part of the domain.  This means that the lower accuracy monotone stencil is needed in this region, which will in turn affect the error in the map.

Despite the degeneracy of this example and the multiple \MA solves required to initialize and solve this problem, the computation requires only  $\bO(M^{1.1})$ time.

\begin{figure}[htdp]
	\centering
	\subfigure[]{\includegraphics[width=.4\textwidth]{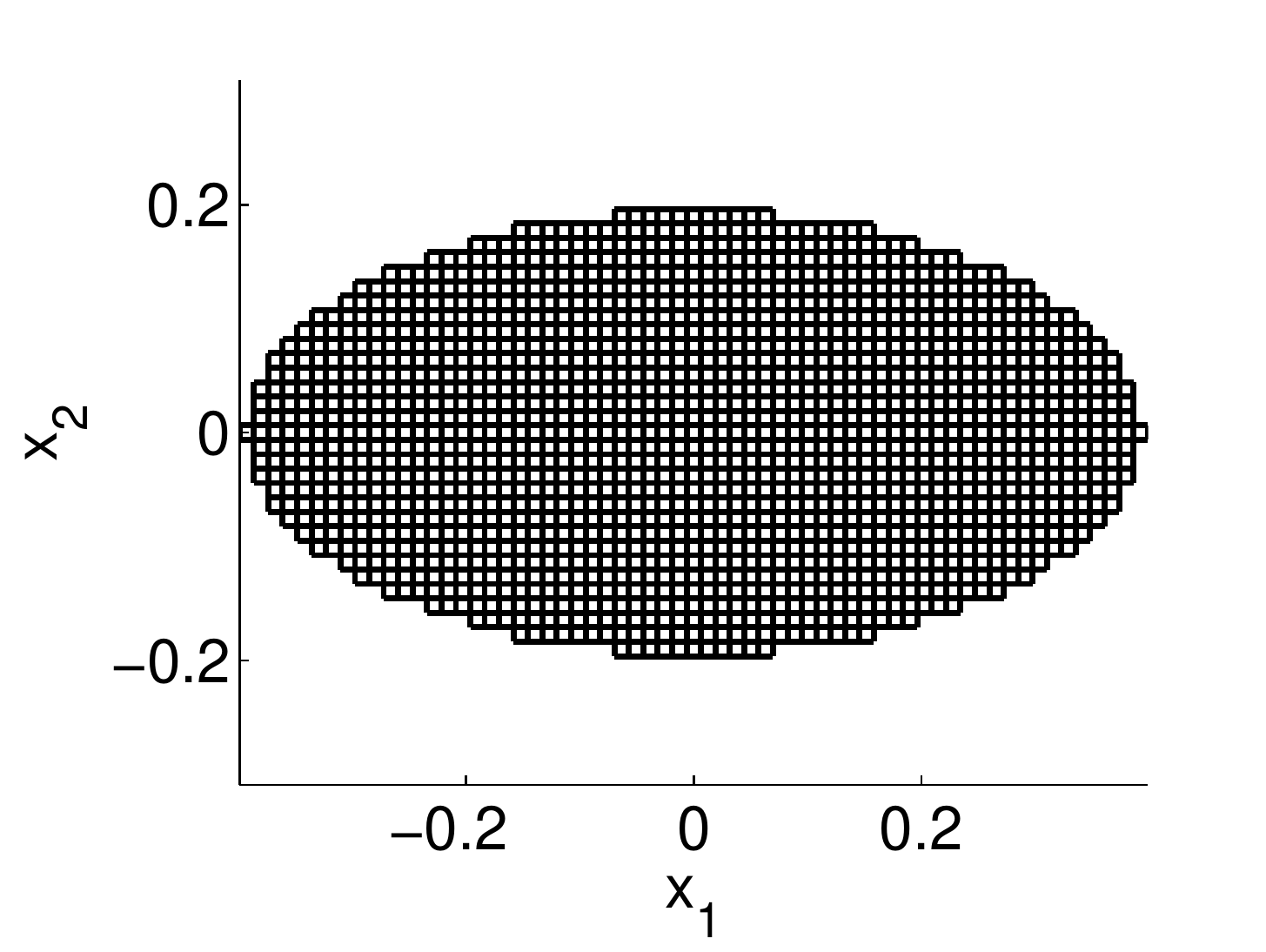}\label{fig:ellipseX}}
        \subfigure[]{\includegraphics[width=.4\textwidth]{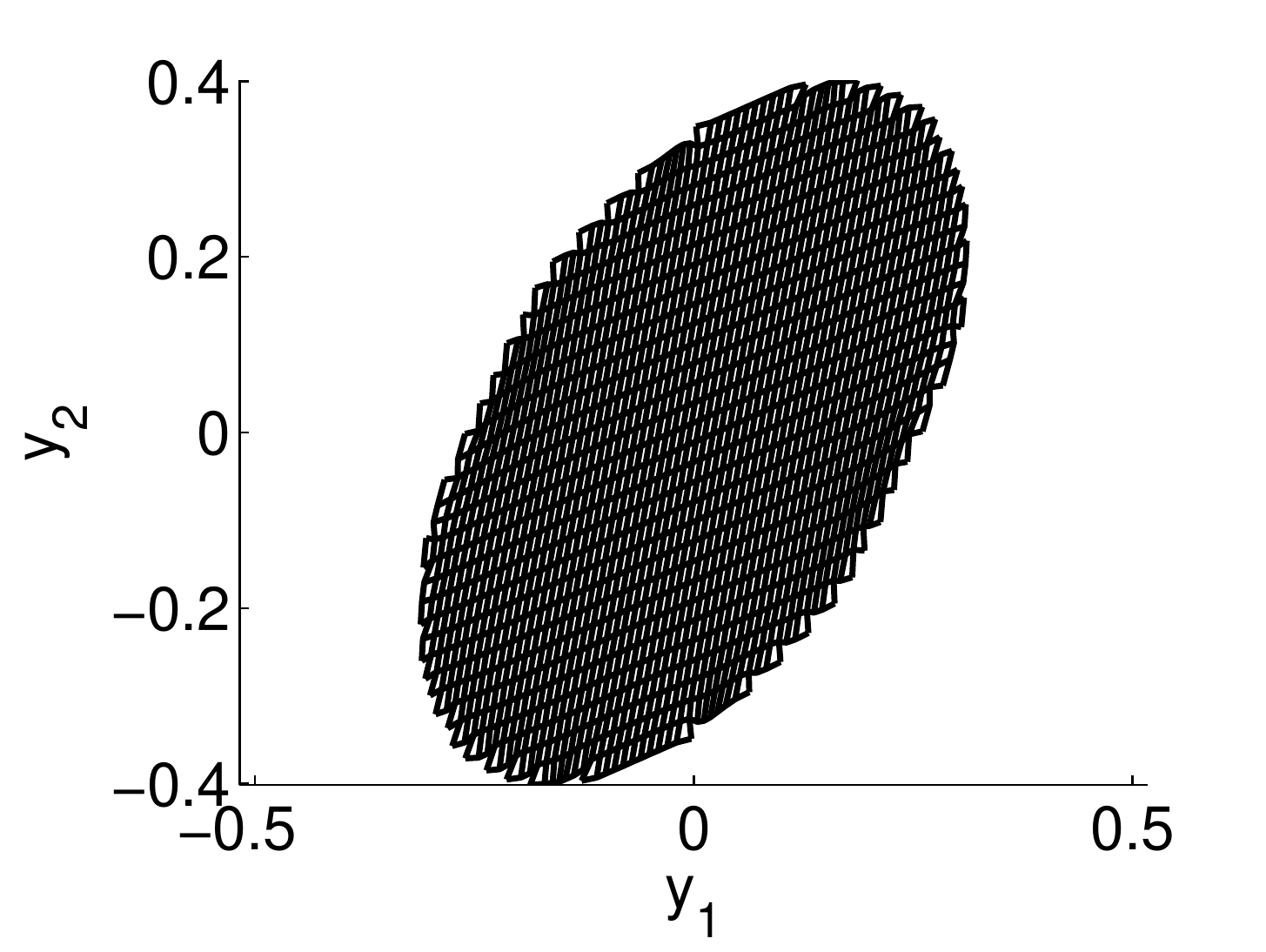}\label{fig:ellipseY}}
  	\vspace*{-12pt}\caption{\subref{fig:ellipseX} A cartesian mesh in the ellipse $X$ and \subref{fig:ellipseY} its image under the gradient map $\nabla u$~(\S\ref{sec:exEllipse}).}
  	\label{fig:ellipses}
\end{figure} 

\begin{table}[htdp]\small
\begin{center}
\begin{tabular}{ccccc}
N  & h & \eqref{eq:MA} Solves & CPU Time (s) & {Maximum Error} \\
\hline
32 & 0.0323 & 4 & 0.7 &  0.0264\\
46 & 0.0222 & 13 & 1.7 &  0.0180\\
64 & 0.0159 &  3 & 1.8 & 0.0152\\
90 & 0.0112 &  6 & 5.5 & 0.0117\\
128 & 0.0079 & 3 & 9.9 & 0.0083\\
182 & 0.0055 & 3 & 25.3 & 0.0060\\
256 & 0.0039 & 2 & 61.9 & 0.0048
\end{tabular}
\end{center}
\caption{Computation time and maximum error for the map between two ellipses (\S\ref{sec:exEllipse}).}
\label{table:ellipses}
\end{table}

\subsection{Mapping from a disconnected region}\label{sec:exSplitCircle}
We now return to the degenerate example considered in~\S\ref{sec:reg}.  This is the problem of mapping the two half-circles
\begin{multline*} X = \{(x_1,x_2)\mid x_1 < -0.1,(x_1+0.1)^2+x_2^2 < 0.3^2 \} \\ \cup \{(x_1,x_2)\mid x_1 > 0.1,(x_1-0.1)^2+x_2^2 < 0.3^2 \}\end{multline*}
onto the circle
\[ Y = \{(y_1,y_2)\mid y_1^2+y_2^2<0.3^2\}.\]

Results are presented in \autoref{table:split} and \autoref{fig:split}.  In this case, the error appears to approach a constant value of around 0.004.  This is not surprising since the monotone scheme is needed in the region where $f$ vanishes or is discontinuous.  The width of the stencil then limits the accuracy of solutions; this point is explained more fully in~\cite{FOTheory}.  The computation time for this very degenerate example is about $\bO(M^{1.3})$.

\begin{figure}[htdp]
	\centering
	\subfigure[]{\includegraphics[width=.4\textwidth]{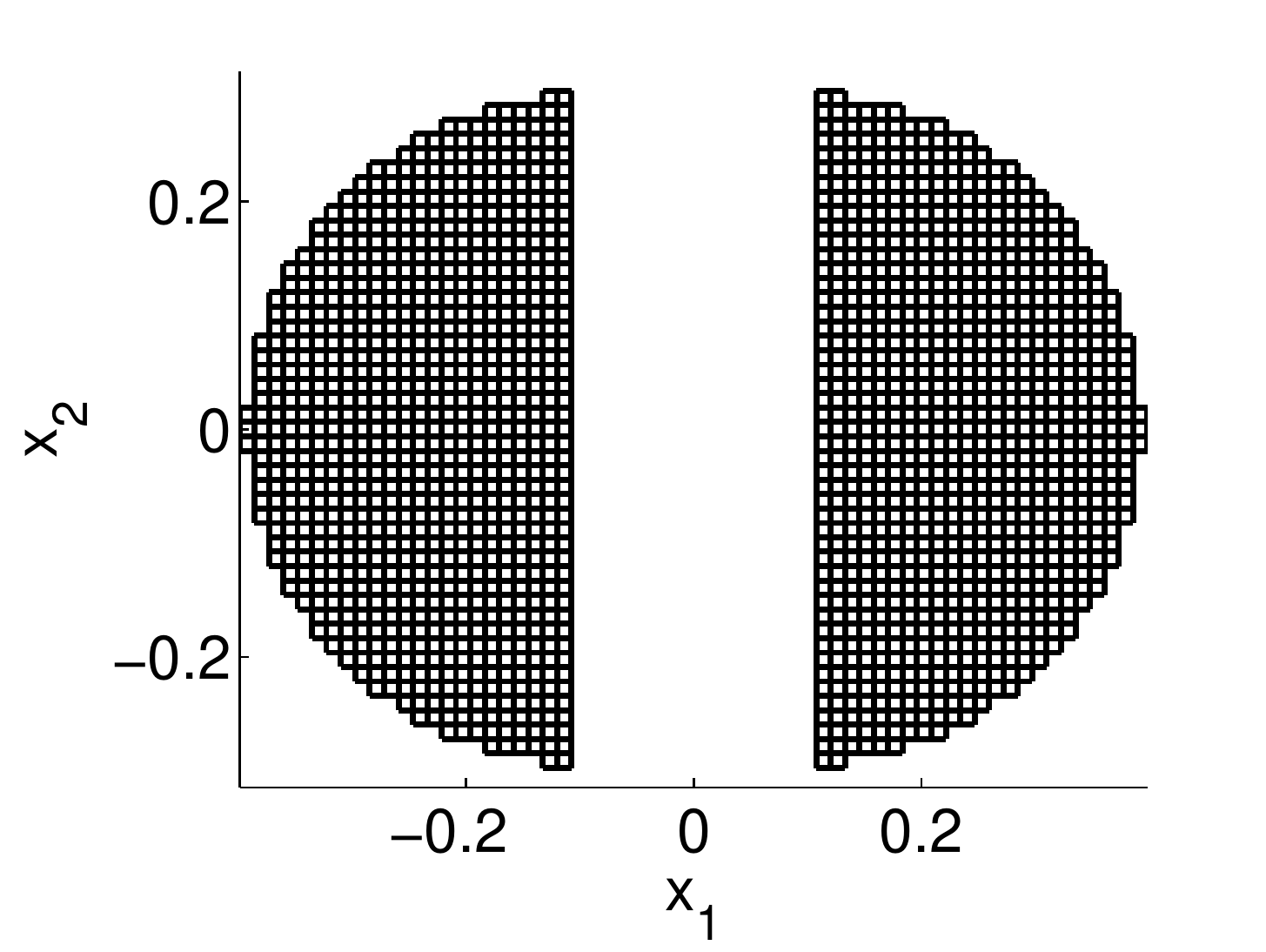}\label{fig:splitX}}
        \subfigure[]{\includegraphics[width=.4\textwidth]{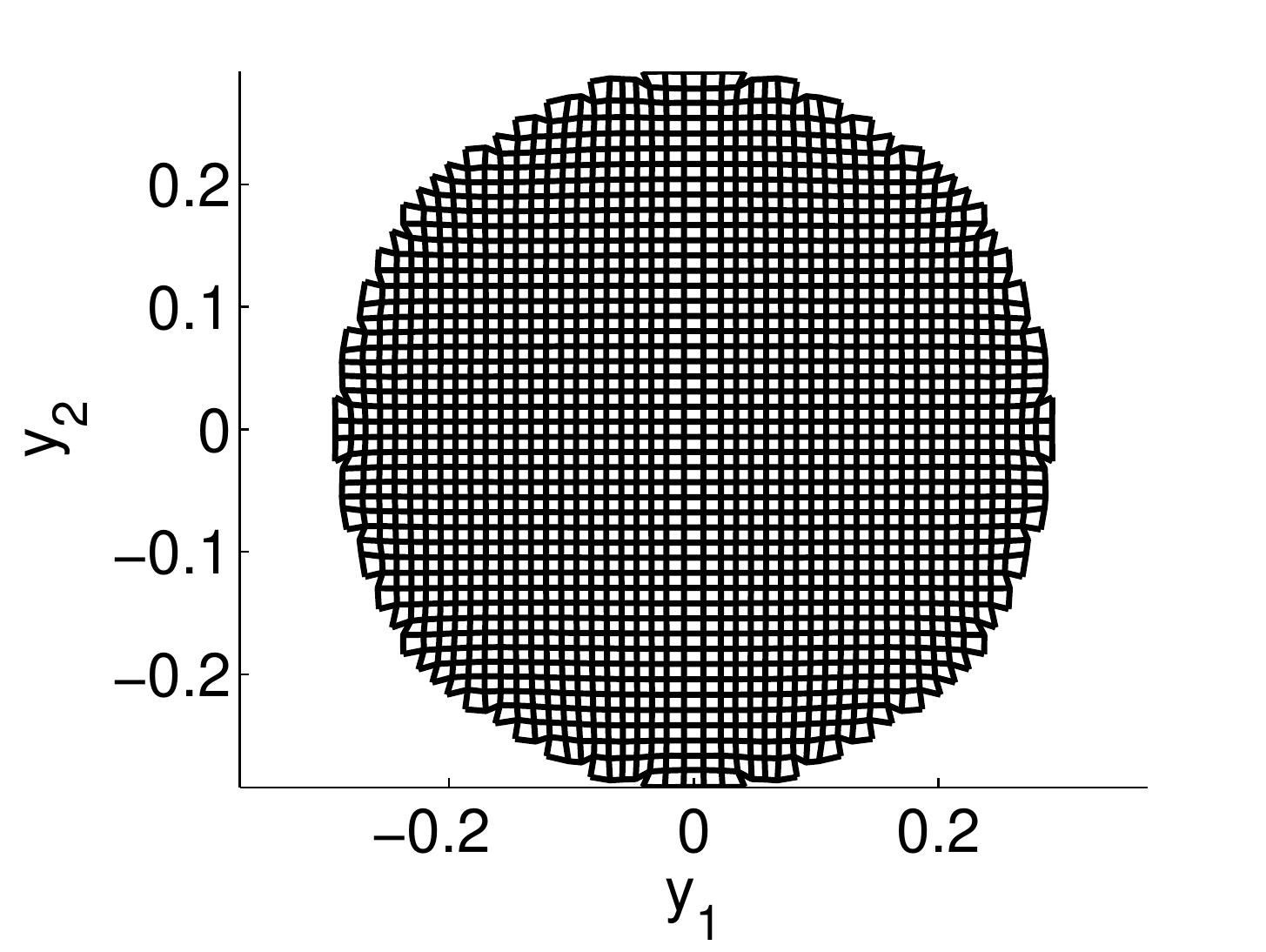}\label{fig:splitY}}
  	\vspace*{-12pt}\caption{\subref{fig:splitX} A cartesian mesh in two half-circles $X$ and \subref{fig:splitY} its image under the gradient map $\nabla u$~(\S\ref{sec:exSplitCircle}).}
  	\label{fig:split}
\end{figure} 

\begin{table}[htdp]\small
\begin{center}
\begin{tabular}{ccccc}
N  & h & \eqref{eq:MA} Solves & CPU Time (s) & {Maximum Error} \\
\hline
32 & 0.0323 & 5 & 0.5 &  0.0171\\
46 & 0.0222 & 2 & 0.5 &  0.0160\\
64 & 0.0159 &  5 & 1.6 & 0.0129\\
90 & 0.0112 &  9 & 6.0 & 0.0082\\
128 & 0.0079 & 5 & 11.8 & 0.0052\\
182 & 0.0055 & 4 & 30.3 & 0.0040\\
256 & 0.0039 & 3 & 66.7 & 0.0038
\end{tabular}
\end{center}
\caption{Computation time and maximum error for the map from two half-circles to a circle (\S\ref{sec:exSplitCircle}).}
\label{table:split}
\end{table}

\subsection{Mapping to a convex polygon}\label{sec:exPol}
Next we consider a map onto a convex polygon $Y$, which has a very non-smooth boundary.   We use the polygon $Y$ with vertices:
\[ (-0.5,-0.3),\,(-0.5,0.4),\,(0,0.5),\,(0.5,0.3),\,(0.3,-0.5). \]
Despite the non-smoothness of $\partial Y$, our method successfully maps the square $(-0.5,0.5)\times(-0.5,0.5)$ into the prescribed polygon, though we do not have an exact solution to compare with.

We also compute this map by solving the inverse problem (mapping the polygon to the square) and inverting the map as in~\S\ref{sec:inverse}.  While no exact solution is available for comparison, we can check the maximum difference between components of the two maps:
\[ \max\{\|u_{x_1} - u_{x_1}^{inv}\|_\infty, \|u_{x_2}-u_{x_2}^{inv}\|_\infty\}. \]

Results are presented in \autoref{table:pol} and \autoref{fig:pol}.  The computation is reasonably efficient, requiring about $\bO(M^{1.2})$ time for both the forward and inverse problem.  We also observe that the agreement between the maps obtained from the forward and inverse approaches improves as we refine the grid.

\begin{figure}[htdp]
	\centering
	\subfigure[]{\includegraphics[width=.4\textwidth]{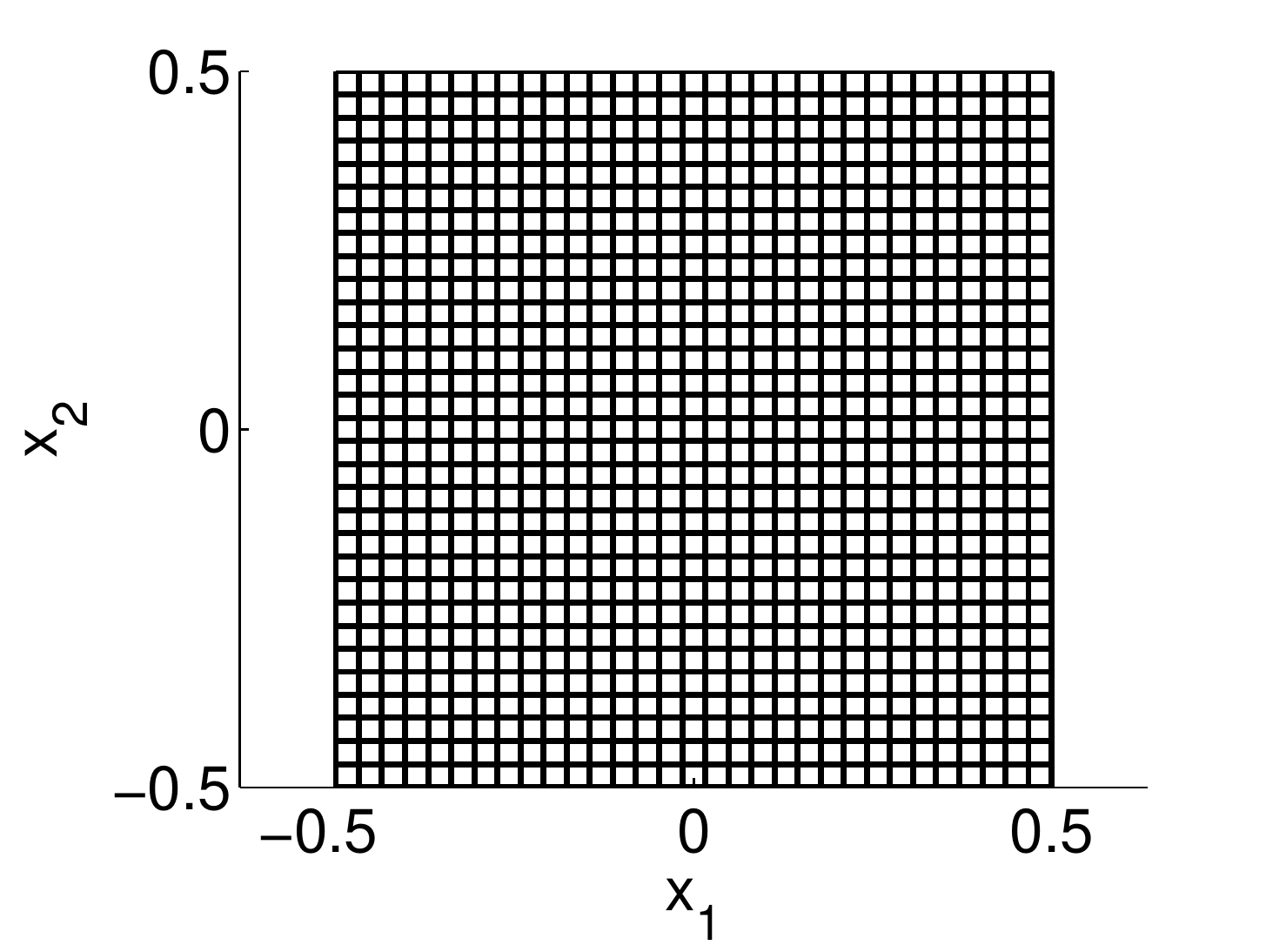}\label{fig:polX}}
        \subfigure[]{\includegraphics[width=.4\textwidth]{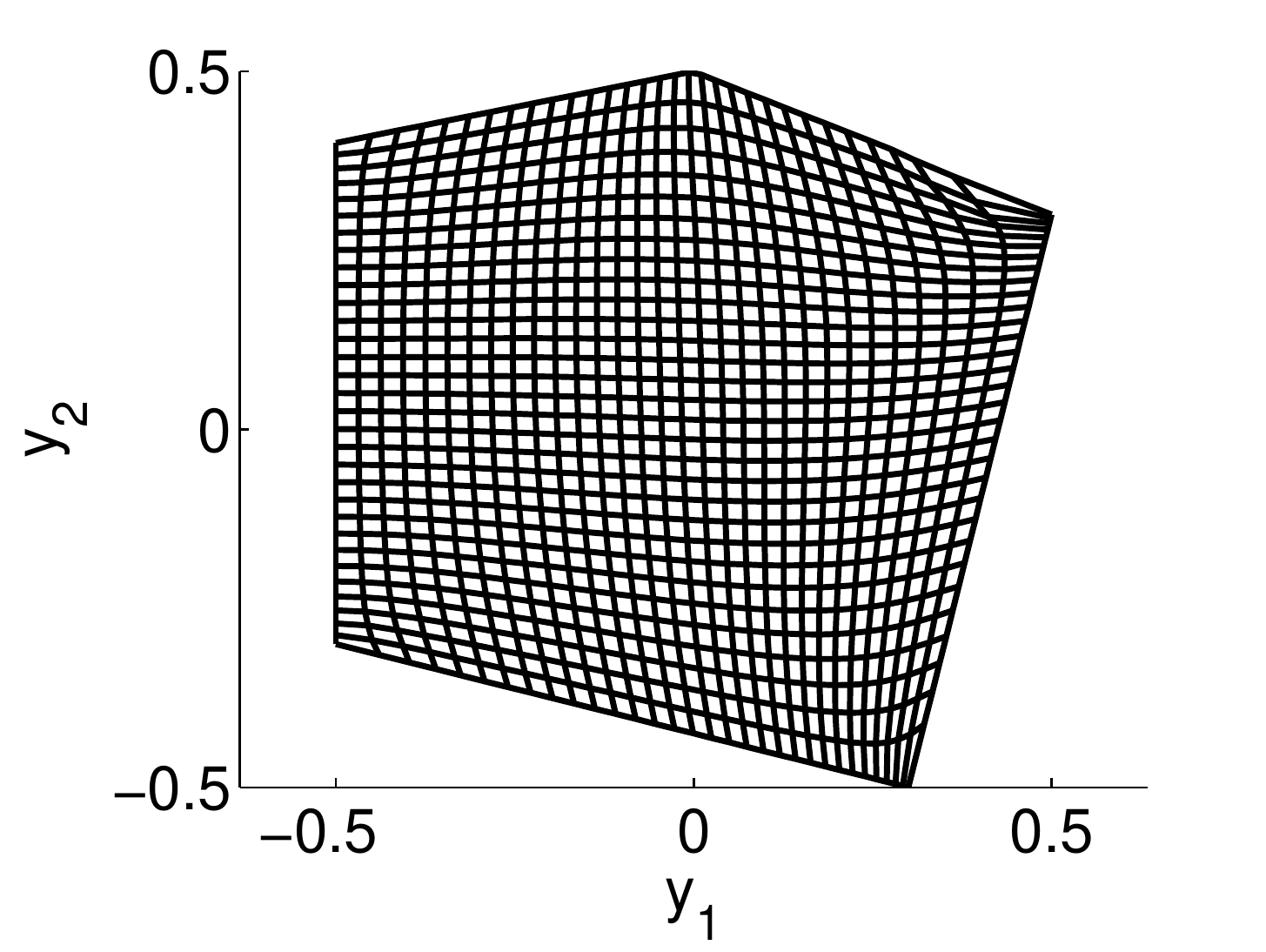}\label{fig:polY}}
  	\vspace*{-12pt}\caption{\subref{fig:polX} A cartesian mesh and \subref{fig:polY} its image under the gradient map $\nabla u$~(\S\ref{sec:exPol}).}
  	\label{fig:pol}
\end{figure} 

\begin{table}[htdp]\small
\begin{center}
\begin{tabular}{cccccc}
 &  \multicolumn{2}{c}{Forward Problem} & \multicolumn{2}{c}{Inverse Problem} \\
N  &  Iterations & Time (s)  & Iterations & Time (s) & Max Difference\\
\hline
32 & 3 & 0.4 & 1 & 0.3 & 0.0397\\
46 & 3 & 0.8 & 1 & 0.7 & 0.0227\\
64 & 3 & 1.5 & 1 & 1.1 & 0.0153\\
90 &  4 & 3.2 & 1 & 2.3 & 0.0119\\
128 & 4 & 8.5 & 1 & 6.2 & 0.0087\\
182 &  4 & 21.0 & 1 & 13.5 & 0.0063\\
256 & 4& 61.8 & 1 & 33.9 & 0.0050\\
362 & 4 & 154.3 & 1 & 92.6 & 0.0044
\end{tabular}
\end{center}
\caption{\MA solves, computation time and maximum difference for a map from square to polygon obtained by a direct solve and by inverting the inverse map (\S\ref{sec:exPol}).}
\label{table:pol}
\end{table}

\subsection{Mapping to a non-convex region}\label{sec:exNonconvex}
Finally, we compute the mapping of the square with constant density $f$ onto a non-convex region given by
\[ Y = \left\{(y_1,y_2) \mid 0 < y_1 < 1,\, 0 < y_2< 1-0.1\sin(2\pi y_1)\right\}. \]
We impose the following periodic density in the region $Y$:
\[  g(y_1,y_2) = 2 + \cos\left(8\pi\sqrt{(y_1-0.5)^2+(y_2-0.5)^2}\right).\]

The results are displayed in \autoref{table:Nonconvex} and \autoref{fig:Nonconvex}.  Despite the non-convexity of $Y$, the method successfully maps the region $X$ into the non-convex region $Y$.  The non-convexity does not appear to affect the computation time at all: the solution time is roughly $\bO(M)$. 

\begin{figure}[htdp]
	\centering
	\subfigure[]{\includegraphics[width=.4\textwidth]{squareX}\label{fig:nonconX}}
        \subfigure[]{\includegraphics[width=.4\textwidth]{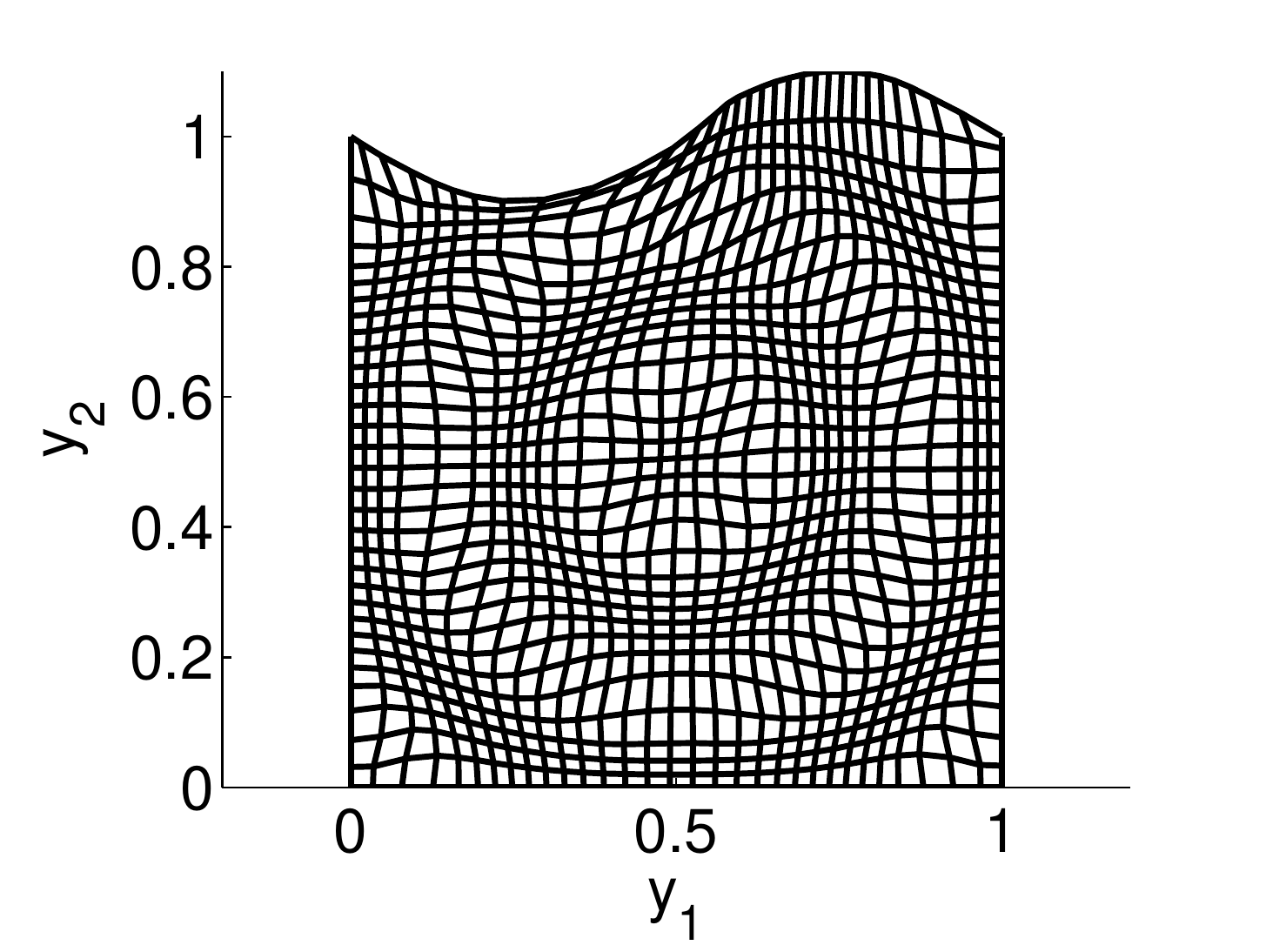}\label{fig:nonconY}}
  	\vspace*{-12pt}\caption{\subref{fig:nonconX} A cartesian mesh and \subref{fig:nonconY} its image under the gradient map $\nabla u$~(\S\ref{sec:exNonconvex}).}
  	\label{fig:Nonconvex}  	
\end{figure} 

\begin{table}[htdp]\small
\begin{center}
\begin{tabular}{ccc}
N   & \eqref{eq:MA} Solves & CPU Time (s)  \\
\hline
32 & 5 & 1.7\\
46 & 4 & 2.2\\
64 & 4 & 5.4\\
90 & 5 & 8.4\\
128 & 5 & 21.4\\
182 & 5 & 41.9\\
256 & 3 & 68.1\\
362 & 3 & 197.4
\end{tabular}
\end{center}
\caption{Computation time for the map onto a non-convex region (\S\ref{sec:exNonconvex}).}
\label{table:Nonconvex}
\end{table}

\section{Conclusions}\label{sec:conclude}
In this paper, we proposed a numerical method for solving the \MA equation with transport boundary conditions.  The method we described requires the solution of a sequence of \MA equations with Neumann boundary conditions.  In practice, we observed that for sufficiently large problems, the number of \MA solves was independent of or even decreased with the size of the problem.  No more than 13 iterations were required in any of the examples considered in this work.

We also presented a discretization of the \MA equation that extends the work of~\cite{FOTheory} to the more general setting where the right-hand side can depend on the gradient of the solution.  We proved that the solution of the discretized system converges to the viscosity solution of the equation and described Newton's method for solving the resulting system efficiently.  

Finally, we provided computational results for several challenging and representative examples including recovering inverse maps, mapping onto an unbounded density, mapping onto a non-convex target set, mapping from a disconnected domain, and mapping between images.  In each case, our method successfully and efficiently ($\bO(M)$-$\bO(M^{1.3})$ time) computed a map into the specified target set.

\newcommand{\etalchar}[1]{$^{#1}$}
\def\polhk#1{\setbox0=\hbox{#1}{\ooalign{\hidewidth
  \lower1.5ex\hbox{`}\hidewidth\crcr\unhbox0}}}


\begin{thebibliography}{CKKSE97}

\bibitem[Amb03]{Ambrosio}
Luigi Ambrosio.
\newblock Lecture notes on optimal transport problems.
\newblock In {\em Mathematical aspects of evolving interfaces ({F}unchal,
  2000)}, volume 1812 of {\em Lecture Notes in Math.}, pages 1--52. Springer,
  Berlin, 2003.

\bibitem[BB00]{BenBren}
Jean-David Benamou and Yann Brenier.
\newblock A computational fluid mechanics solution to the {M}onge-{K}antorovich
  mass transfer problem.
\newblock {\em Numer. Math.}, 84(3):375--393, 2000.

\bibitem[Ber03]{Bertsekas}
Dimitri~P. Bertsekas.
\newblock {\em Convex analysis and optimization}.
\newblock Athena Scientific, Belmont, MA, 2003.
\newblock With Angelia Nedi\'c and Asuman E. Ozdaglar.

\bibitem[BFO10]{BeFrObMA}
Jean-David Benamou, Brittany~D. Froese, and Adam~M. Oberman.
\newblock Two numerical methods for the elliptic {M}onge-{A}mp\`ere equation.
\newblock {\em ESAIM: Math. Model. Numer. Anal.}, 44(4), 2010.

\bibitem[BP07]{BarPrigL1Transport}
John~W. Barrett and Leonid Prigozhin.
\newblock A mixed formulation of the {M}onge-{K}antorovich equations.
\newblock {\em M2AN Math. Model. Numer. Anal.}, 41(6):1041--1060, 2007.

\bibitem[BP09]{BarPrigL1PartialTransport}
John~W. Barrett and Leonid Prigozhin.
\newblock Partial {$L^1$} {M}onge-{K}antorovich problem: variational
              formulation and numerical approximation.
\newblock {\em Interfaces Free Bound.}, 11(2):201--238, 2009.

\bibitem[BS91]{BSnum}
Guy Barles and Panagiotis~E. Souganidis.
\newblock Convergence of approximation schemes for fully nonlinear second order
  equations.
\newblock {\em Asymptotic Anal.}, 4(3):271--283, 1991.

\bibitem[BW09]{Budd}
C.~J. Budd and J.~F. Williams.
\newblock Moving mesh generation using the parabolic {M}onge-{A}mp\`ere
  equation.
\newblock {\em SIAM J. Sci. Comput.}, 31(5):3438--3465, 2009.

\bibitem[Caf92a]{CafBdyReg}
Luis~A. Caffarelli.
\newblock Boundary regularity of maps with convex potentials.
\newblock {\em Comm. Pure Appl. Math.}, 45(9):1141--1151, 1992.

\bibitem[Caf92b]{CafRegMap}
Luis~A. Caffarelli.
\newblock The regularity of mappings with a convex potential.
\newblock {\em J. Amer. Math. Soc.}, 5(1):99--104, 1992.

\bibitem[Caf96]{CafBdyReg2}
Luis~A. Caffarelli.
\newblock Boundary regularity of maps with convex potentials. {II}.
\newblock {\em Ann. of Math. (2)}, 144(3):453--496, 1996.

\bibitem[Cen10]{brainweb}
McConnell Brain~Imaging Center.
\newblock Brainweb: Simulated brain database, November 2010.
\newblock http://www.bic.mni.mcgill.ca/brainweb.

\bibitem[CIL92]{CIL}
Michael~G. Crandall, Hitoshi Ishii, and Pierre-Louis Lions.
\newblock User's guide to viscosity solutions of second order partial
  differential equations.
\newblock {\em Bull. Amer. Math. Soc. (N.S.)}, 27(1):1--67, 1992.

\bibitem[CKKSE97]{CocoscoBrains}
C.~A. Cocosco, V.~Kollokian, Kwan~R. K.-S., and A.~C. Evans.
\newblock Brainweb: Online interface to a 3d mri simulated brain database.
\newblock In {\em NeuroImage}, volume~5, 1997.

\bibitem[CZK{\etalchar{+}}98]{CollinsBrains}
D.~L. Collins, A.~P. Zijenbos, N.~J. Kollokian, J.and~Sled, N.~J. Kabani, C.~J.
  Holmes, and A.~C. Evans.
\newblock Design and construction of a realistic digital brain phantom.
\newblock {\em IEEE Transactions on Medical Imaging}, 17(3):463--468, 1998.

\bibitem[DCF{\etalchar{+}}08]{Delzanno}
G.~L. Delzanno, L.~Chac{\'o}n, J.~M. Finn, Y.~Chung, and G.~Lapenta.
\newblock An optimal robust equidistribution method for two-dimensional grid
  adaptation based on {M}onge-{K}antorovich optimization.
\newblock {\em J. Comput. Phys.}, 227(23):9841--9864, 2008.

\bibitem[DG06]{DGaug}
E.~J. Dean and R.~Glowinski.
\newblock An augmented {L}agrangian approach to the numerical solution of the
  {D}irichlet problem for the elliptic {M}onge-{A}mp\`ere equation in two
  dimensions.
\newblock {\em Electron. Trans. Numer. Anal.}, 22:71--96 (electronic), 2006.

\bibitem[DG08]{DGnum2008}
Edward~J. Dean and Roland Glowinski.
\newblock On the numerical solution of the elliptic {M}onge-{A}mp\`ere equation
  in dimension two: a least-squares approach.
\newblock In {\em Partial differential equations}, volume~16 of {\em Comput.
  Methods Appl. Sci.}, pages 43--63. Springer, Dordrecht, 2008.

\bibitem[Eva99]{EvansSurvey}
Lawrence~C. Evans.
\newblock Partial differential equations and {M}onge-{K}antorovich mass
  transfer.
\newblock In {\em Current developments in mathematics, 1997 (Cambridge, MA)},
  pages 65--126. Int. Press, Boston, MA, 1999.

\bibitem[FDC08]{DelzannoGrid}
J.~M. Finn, G.~L. Delzanno, and L.~Chac{\'o}n.
\newblock Grid generation and adaptation by {M}onge-{K}antorovich optimization
  in two and three dimensions.
\newblock In {\em Proceedings of the 17th International Meshing Roundtable},
  pages 551--568, 2008.

\bibitem[FMMS02]{FrischUniv}
Uriel Frisch, Sabino Matarrese, Roya Mohayaee, and Andrei Sobolevski.
\newblock A reconstruction of the initial conditions of the universe by optimal
  mass transportation.
\newblock {\em Nature}, 417, 2002.

\bibitem[FN09a]{FengMA}
Xiaobing Feng and Michael Neilan.
\newblock Mixed finite element methods for the fully nonlinear
  {M}onge-{A}mp\`ere equation based on the vanishing moment method.
\newblock {\em SIAM J. Numer. Anal.}, 47(2):1226--1250, 2009.

\bibitem[FN09b]{FengFully}
Xiaobing Feng and Michael Neilan.
\newblock Vanishing moment method and moment solutions for fully nonlinear
  second order partial differential equations.
\newblock {\em J. Sci. Comput.}, 38(1):74--98, 2009.

\bibitem[FO11a]{FOTheory}
Brittany~D. Froese and Adam~M. Oberman.
\newblock Convergent finite difference solvers for viscosity solutions of the
  elliptic {M}onge-{A}mp\`ere equation in dimensions two and higher.
\newblock {\em SIAM J. Numer. Anal.}, 49(4):1692--1714, 2011.

\bibitem[FO11b]{FONum}
Brittany~D. Froese and Adam~M. Oberman.
\newblock Fast finite difference solvers for singular solutions of the elliptic
  {M}onge-{A}mp\`ere equation.
\newblock {\em J. Comput. Phys.}, 230(3):818--834, 2011.

\bibitem[Glo09]{GlowinksiICIAM}
Roland Glowinski.
\newblock Numerical methods for fully nonlinear elliptic equations.
\newblock In Rolf Jeltsch and Gerhard Wanner, editors, {\em 6th International
  Congress on Industrial and Applied Mathermatics, ICIAM 07, Invited Lectures},
  pages 155--192, 2009.

\bibitem[GO03]{GlimmOlikerReflectorDesign}
T.~Glimm and V.~Oliker.
\newblock Optical design of single reflector systems and the
  {M}onge-{K}antorovich mass transfer problem.
\newblock {\em J. Math. Sci. (N. Y.)}, 117(3):4096--4108, 2003.
\newblock Nonlinear problems and function theory.

\bibitem[GO04]{GlimmOliker2Reflectors}
Tilmann Glimm and Vladimir Oliker.
\newblock Optical design of two-reflector systems, the {M}onge-{K}antorovich
  mass transfer problem and {F}ermat's principle.
\newblock {\em Indiana Univ. Math. J.}, 53(5):1255--1277, 2004.

\bibitem[Gut01]{Gutierrez}
Cristian~E. Guti{\'e}rrez.
\newblock {\em The {M}onge-{A}mp\`ere equation}.
\newblock Progress in Nonlinear Differential Equations and their Applications,
  44. Birkh\"auser Boston Inc., Boston, MA, 2001.

\bibitem[HRT10]{HaberTransport}
Eldad Haber, Tauseef Rehman, and Allen Tannenbaum.
\newblock An efficient numerical method for the solution of the {$L_2$} optimal
  mass transfer problem.
\newblock {\em SIAM J. Sci. Comput.}, 32(1):197--211, 2010.

\bibitem[HTK01]{Haker}
Steven Haker, Allen Tannenbaum, and Ron Kikinis.
\newblock Mass preserving mappings and image registration.
\newblock In {\em MICCAI '01: Proceedings of the 4th International Conference
  on Medical Image Computing and Computer-Assisted Intervention}, pages
  120--127, London, UK, 2001. Springer-Verlag.

\bibitem[HZTA04]{HakerRegistration}
Steven Haker, Lei Zhu, Allen Tannenbaum, and Sigurd Angenent.
\newblock Optimal mass transport for registration and warping.
\newblock {\em Int. J. Comput. Vision}, 60(3):225--240, 2004.

\bibitem[Kan42]{Kantorovich1}
L.~V. Kantorovich.
\newblock On the transfer of masses.
\newblock {\em Dokl. Akad. Nauk. SSSR}, 37(7--8):227--229, 1942.

\bibitem[Kan48]{Kantorovich2}
L.~V. Kantorovich.
\newblock On a problem of {M}onge.
\newblock {\em Uspekhi Mat. Nauk.}, 3(2):225–-226, 1948.

\bibitem[Kis94]{KiselevProject}
Yu.~N. Kiselev.
\newblock Algorithms for the projection of a point onto an ellipsoid.
\newblock {\em Liet. Mat. Rink.}, 34(2):174--196, 1994.

\bibitem[Lio85]{LionsNeumannHJ}
P.~L. Lions.
\newblock Neumann type boundary conditions for {H}amilton-{J}acobi equations.
\newblock {\em Duke Math. J.}, 52(3):793–-820, 1985.

\bibitem[LTU86]{LionsNeumannMA}
P.-L. Lions, N.~S. Trudinger, and J.~I.~E. Urbas.
\newblock The {N}eumann problem for equations of {M}onge-{A}mp\`ere type.
\newblock {\em Comm. Pure Appl. Math.}, 39(4):539--563, 1986.

\bibitem[MO04]{MOEllipse}
Robert~J. McCann and Adam~M. Oberman.
\newblock Exact semi-geostrophic flows in an elliptical ocean basin.
\newblock {\em Nonlinearity}, 17(5):1891--1922, 2004.

\bibitem[MW53]{MotzkinWasow}
T.~S. Motzkin and W. Wasow.
\newblock On the approximation of linear elliptic differential equations by difference equations with positive coefficients.
\newblock {\em J. Math. Phys.}, 31:253--259, 1953.

\bibitem[Obe06]{ObermanSINUM}
Adam~M. Oberman.
\newblock Convergent difference schemes for degenerate elliptic and parabolic
  equations: {H}amilton-{J}acobi equations and free boundary problems.
\newblock {\em SIAM J. Numer. Anal.}, 44(2):879--895 (electronic), 2006.

\bibitem[Obe08]{ObermanEigenvalues}
Adam~M. Oberman.
\newblock Wide stencil finite difference schemes for the elliptic
  {M}onge-{A}mp\`ere equation and functions of the eigenvalues of the
  {H}essian.
\newblock {\em Discrete Contin. Dyn. Syst. Ser. B}, 10(1):221--238, 2008.

\bibitem[OP88]{olikerprussner88}
V.~I. Oliker and L.~D. Prussner.
\newblock On the numerical solution of the equation $(\partial^2z/\partial
  x^2)(\partial^2z/\partial y^2)-(\partial^2z/\partial x\partial y)^2=f$ and
  its discretizations, {I}.
\newblock {\em Numer. Math.}, 54(3):271--293, 1988.

\bibitem[Pog71]{Pogorelov}
A.~V. Pogorelov.
\newblock The {D}irichlet problem for the multidimensional analogue of the
  {M}onge-{A}mp\`ere equation.
\newblock {\em Dokl. Akad. Nauk SSSR}, 201:790--793, 1971.

\bibitem[Roc66]{Rockafellar}
R.~T. Rockafellar.
\newblock Characterization of the subdifferentials of convex functions.
\newblock {\em Pacific J. Math.}, 17:497--510, 1966.

\bibitem[TW09]{TrudWang2ndBVP}
Neil~S. Trudinger and Xu-Jia Wang.
\newblock On the second boundary value problem for {M}onge-{A}mp\`ere type
  equations and optimal transportation.
\newblock {\em Ann. Sc. Norm. Super. Pisa Cl. Sci. (5)}, 8(1):143--174, 2009.

\bibitem[Urb97]{Urbas2ndBVP}
John Urbas.
\newblock On the second boundary value problem for equations of
  {M}onge-{A}mp\`ere type.
\newblock {\em J. Reine Angew. Math.}, 487:115--124, 1997.

\bibitem[uRHP{\etalchar{+}}09]{RehmanRegistration}
T.~ur~Rehman, E.~Haber, G.~Pryor, J.~Melonakos, and A.~Tannenbaum.
\newblock 3{D} nonrigid registration via optimal mass transport on the {GPU}.
\newblock {\em Med Image Anal}, 13(6):931--40, 12 2009.

\bibitem[Vil03]{Villani}
C{\'e}dric Villani.
\newblock {\em Topics in optimal transportation}, volume~58 of {\em Graduate
  Studies in Mathematics}.
\newblock American Mathematical Society, Providence, RI, 2003.

\end{thebibliography}
\end{document}